\documentclass[12pt,leqno]{article}%draftmode
\usepackage{amsfonts}
\usepackage{amsmath, amsthm, amsfonts, amssymb,xcolor}
\usepackage{mathrsfs}
\usepackage{color}
\usepackage[notcite,notref]{showkeys}

\usepackage[normalem]{ulem}

\usepackage{hyperref}
\newtheorem{thm}{Theorem}[section]

\newtheorem{lem}[thm]{Lemma}

\newtheorem{exa}[thm]{Example}
\newtheorem{rem}[thm]{Remark}
\theoremstyle{definition}

\newcommand{\scr}[1]{\mathscr #1}
\definecolor{wco}{rgb}{0.5,0.2,0.3}

\numberwithin{equation}{section} \theoremstyle{remark}

\newcommand{\ua}{\uparrow}

\title{{\bf
  Uniform Hypercontractivity and Log-Sobolev Inequality and Applications for Non-equilibrium Kinetic System}\footnote{
Xing Huang  is Supported in
 part by  National Key R\&D Program of China (No. 2022YFA1006000) and NNSFC (12271398). Eva Kopfer is supported by the German Research Foundation
 		through the Hausdorff Center for Mathematics and the Collaborative Research Center 1060. Panpan Ren is supported by NNSFC (12301180) and Research Center for Nonlinear Analysis at The Hong Kong Polytechnic University. } }
\author{
{\bf   Xing Huang$^{(a)}$,  Eva Kopfer$^{(b)}$, Pierre Monmarch\'{e}$^{(c)}$, Panpan Ren$^{(d)}$}\\
\footnotesize{$^{a)}$  Center for Applied Mathematics, Tianjin University, Tianjin 300072, China}\\
\footnotesize{$^{b)}$ Institut f\"ur Angewandte Mathematik, Universit\"at Bonn, Endenicher Allee 60,  Bonn,  Germany }\\
\footnotesize{$^{c)}$ LJLL, Sorbonne Universit\'{e}, 4 place Jussieu, 75005, Paris, France }\\
\footnotesize{$^{d)}$ Department of Mathematics, City University of  Hong Kong, Tat Chee Avenue, Hong Kong,  China }\\
\footnotesize{ xinghuang@tju.edu.cn;  eva.kopfer@iam.uni-bonn.de; pierre.monmarche@sorbonne-universite.fr; panparen@cityu.edu.hk}}
\begin{document}
\allowdisplaybreaks
\def\R{\mathbb R}  \def\ff{\frac} \def\ss{\sqrt} \def\B{\mathbf
B} \def\W{\mathbb W}
\def\N{\mathbb N} \def\kk{\kappa} \def\m{{\bf m}}
\def\ee{\varepsilon}\def\ddd{D^*}
\def\dd{\delta} \def\DD{\Delta} \def\vv{\varepsilon} \def\rr{\rho}
\def\<{\langle} \def\>{\rangle} \def\GG{\Gamma} \def\gg{\gamma}
  \def\nn{\nabla} \def\pp{\partial} \def\E{\mathbb E}
\def\d{\text{\rm{d}}} \def\bb{\beta} \def\aa{\alpha} \def\D{\scr D}
  \def\si{\sigma} \def\ess{\text{\rm{ess}}}
\def\beg{\begin} \def\beq{\begin{equation}}  \def\F{\scr F}
\def\Ric{\text{\rm{Ric}}} \def\Hess{\text{\rm{Hess}}}
\def\e{\text{\rm{e}}} \def\ua{\underline a} \def\OO{\Omega}  \def\oo{\omega}
 \def\tt{\tilde} \def\Ric{\text{\rm{Ric}}}
\def\cut{\text{\rm{cut}}} \def\P{\mathbb P} \def\ifn{I_n(f^{\bigotimes n})}
\def\C{\scr C}      \def\aaa{\mathbf{r}}     \def\r{r}
\def\gap{\text{\rm{gap}}} \def\prr{\pi_{{\bf m},\varrho}}  \def\r{\mathbf r}
\def\Z{\mathbb Z} \def\vrr{\varrho} \def\ll{\lambda}
\def\L{\scr L}\def\Tt{\tt} \def\TT{\tt}\def\II{\mathbb I}
\def\i{{\rm in}}\def\Sect{{\rm Sect}}  \def\H{\mathbb H}
\def\M{\scr M}\def\Q{\mathbb Q} \def\texto{\text{o}} \def\LL{\Lambda}
\def\Rank{{\rm Rank}} \def\B{\scr B} \def\i{{\rm i}} \def\HR{\hat{\R}^d}
\def\to{\rightarrow}\def\l{\ell}\def\iint{\int}
\def\EE{\scr E}\def\Cut{{\rm Cut}}
\def\A{\scr A} \def\Lip{{\rm Lip}}
\def\BB{\scr B}\def\Ent{{\rm Ent}}\def\L{\scr L}
\def\R{\mathbb R}  \def\ff{\frac} \def\ss{\sqrt} \def\B{\mathbf
B}
\def\N{\mathbb N} \def\kk{\kappa} \def\m{{\bf m}}
\def\dd{\delta} \def\DD{\Delta} \def\vv{\varepsilon} \def\rr{\rho}
\def\<{\langle} \def\>{\rangle} \def\GG{\Gamma} \def\gg{\gamma}
  \def\nn{\nabla} \def\pp{\partial} \def\E{\mathbb E}
\def\d{\text{\rm{d}}} \def\bb{\beta} \def\aa{\alpha} \def\D{\scr D}
  \def\si{\sigma} \def\ess{\text{\rm{ess}}}
\def\beg{\begin} \def\beq{\begin{equation}}  \def\F{\scr F}
\def\Ric{\text{\rm{Ric}}} \def\Hess{\text{\rm{Hess}}}
\def\e{\text{\rm{e}}} \def\ua{\underline a} \def\OO{\Omega}  \def\oo{\omega}
 \def\tt{\tilde} \def\Ric{\text{\rm{Ric}}}
\def\cut{\text{\rm{cut}}} \def\P{\mathbb P} \def\ifn{I_n(f^{\bigotimes n})}
\def\C{\scr C}      \def\aaa{\mathbf{r}}     \def\r{r}
\def\gap{\text{\rm{gap}}} \def\prr{\pi_{{\bf m},\varrho}}  \def\r{\mathbf r}
\def\Z{\mathbb Z} \def\vrr{\varrho} \def\ll{\lambda}
\def\L{\scr L}\def\Tt{\tt} \def\TT{\tt}\def\II{\mathbb I}
\def\i{{\rm in}}\def\Sect{{\rm Sect}}  \def\H{\mathbb H}
\def\M{\scr M}\def\Q{\mathbb Q} \def\texto{\text{o}} \def\LL{\Lambda}
\def\Rank{{\rm Rank}} \def\B{\scr B} \def\i{{\rm i}} \def\HR{\hat{\R}^d}
\def\to{\rightarrow}\def\l{\ell}\def\BB{\mathbb B}
\def\8{\infty}\def\I{1}\def\U{\scr U} \def\n{{\mathbf n}}\def\v{V}
\maketitle

\begin{abstract}
 In this paper, we derive exponential ergodicity in relative entropy for general kinetic SDEs under a partially dissipative condition. It covers non-equilibrium situations where the forces are not of gradient type and the invariant measure does not have an explicit density, extending previous results set in the equilibrium case. The key argument is to establish the hypercontractivity of the associated semigroup, which follows from its hyperboundedness and its $L^2$-exponential ergodicity.
Moreover, we obtain exponential ergodicity in the $L^2$-Wasserstein distance by combining Talagrand's inequality with a log-Harnack inequality. These results are further extended to the McKean-Vlasov setting and to the associated mean-field interacting particle systems  under a suitable smallness condition on the distribution dependence.  In the latter case, the convergence rates are uniform in the number of particles and we also establish the hypercontractivity and log-Sobolev inequality uniform in the number of particles using a new approach based on Girsanov's transform and hypercontractivity.
 \end{abstract}


 \tableofcontents
\section{Introduction}

 The (kinetic) Langevin equation is a central model of statistical physics and molecular dynamics. It describes the motion of a classical particle subject to external forces and coupled with a Brownian heat bath. In the so-called equilibrium case, the forces derive from a potential and the Langevin equation reads
\begin{align}\label{Degthv}\left\{
  \begin{array}{ll}
    \d X_t=Y_t\d t,  \\
    \d Y_t=-\gamma Y_t\d t - \nabla V(X_t)\d t + \sqrt{2\gamma \beta^{-1}}\d W_t\,,
  \end{array}
\right.
\end{align}
where $\gamma>0$ is the friction parameter, $\beta^{-1}$ is the inverse temperature, $V$ is the external (or confining) potential and $W$ is a multi-dimensional Brownian motion. Here, $X\in\R^d$ and $Y\in\R^d$ respectively stands for the position and velocity of the particle. This equation has been extensively studied and we refer e.g. to \cite{LS} and references within for general considerations. Specifically, in the present work, we are interested in the long-time relaxation to a statistical steady state for the process. Under mild conditions on $V$, \eqref{Degthv} is known to be ergodic with respect to the Gibbs measure $\mu$ with density proportional to $\e^{- \beta H}$ with $H(x,y)=V(x) + |y|^2/2$ (assuming $e^{-\beta H} \in L^1$). Exponential ergodicity  under suitable conditions was first established in \cite{Talay,MSH}, using Harris theorem.

 The first exponential convergence result stated in terms of relative entropy has been obtained in the seminal work \cite{Villani} of Villani. It introduced the notion of hypocoercivity which has been developed since then in various directions, see \cite{DMS,M24,AAMN,CLW,BFLS,BLW} and references therein on this general topic. More specifically, \cite[Theorem 39]{Villani} gives an exponential decay in relative entropy along~\eqref{Degthv} provided $V\in C^2(\R^d)$ and has a bounded Hessian, and $\mu$ satisfies a so-called log-Sobolev inequality (see \eqref{logsob1} below)  with respect to the Dirichlet form $\scr E(f,g)=\mu(\nabla f\cdot \nabla g)$, where $\nabla $ is the gradient on $\R^{2d}$.

 Apart from its intrinsic role in statistical physics,  a key advantage of the relative entropy with respect to total variation, $V$-norms or $L^2$ norms is that it scales well with dimension, in particular for mean-field systems and thus, by extension, is suitable for the non-linear limits of the latter; moreover it also behaves well with respect to discretization schemes.   This explains that obtaining quantitative convergence rates for the Langevin equation in relative entropy, in particular in mean-field cases~\cite{CLRW,CRW,NWS}, is still a very active topic, considering heavy tail distributions~\cite{EH}, non-elliptic diffusion~\cite{HKR25}, local convergence~\cite{MR24} and time-discretization~\cite{MCCFBI,CDMS,MS,VW}, motivated in particular by the mean-field analysis of high-dimensionl algorithms~\cite{Chizat,FW,SWN}.

However, the result in \cite[Theorem 39]{Villani}, along with nearly all subsequent works, relies on having an explicit representation of the invariant probability measure $\mu$.  An exception arises in the so-called uniformly dissipative case, which is a very restrictive situation where the drift of the SDE induces a deterministic contraction along the stochastic flow~\cite{Baudoin,Menegaki,M23}. In particular, there are many tools to establish  log-Sobolev inequalities \cite{BGL}, but most of them require an explicit expression, or are restricted to the uniformly-dissipative case as the Bakry-\'Emery criterion (which in its classical form is written for explicit invariant measures but in fact apply more generally~\cite{M23}).

 Situations where $\mu$ has no explicit form are known as non-equilibrium models \cite{GOSSS}. This is for instance the case if the conservative force $-\nabla V$ is replaced by a non-gradient force field (as in \cite{EM,IOG19,IOG21}), or if the temperature $\beta^{-1}$ is not constant among different coordinates (as in \cite{BLRB,LO}). In contrast to the equilibrium setting, the position and velocity are not independent under $\mu$ in these non-equilibrium situations \cite{MR}. There are also non-equilibrium examples in mean-field situations, for instance when the interaction between particles is not symmetric as in some plasma models \cite{CH,M23bis}, mean-field adversarial games \cite{LM,WC} or other situations with multi-type particles~\cite{MACF}.

The question of relaxation to a steady state for kinetic non-equilibrium models involves phenomena which can be counter-intuitive based on intuitions formed on elliptic or equilibrium situations. For instance, for a non-equilibrium elliptic model of the form
\[\d X_t = F(X_t) \d t + \d B_t, \]
exponential ergodicity is standard under a Lyapunov condition of the form $x\cdot F(x) \leqslant - \lambda |x|^2$ outside a compact set for some $\lambda>0$, meaning that the force $F$ has a sufficient confining effect. In the equilibrium kinetic case~\eqref{Degthv}, similarly, if $V$ is convex, so that a similar confining condition holds for $F(x)= -\nabla V(x)$, exponential ergodicity is known. However, in non-equilibrium non-elliptic cases, it has been observed e.g. in~\cite{H} that a strong confining force can lead to highly oscillatory behaviors which eventually prevents exponential ergodicity.

Although recent works such as \cite{MW,D2023,M25} have established exponential ergodicity in $L^2$-distances under potentially restrictive conditions, proving exponential convergence in relative entropy for non-equilibrium kinetic stochastic differential equations (SDEs) with only partially dissipative drift has remained an open problem for a long time. This work addresses that problem and further extends the analysis to nonlinear McKean-Vlasov equations and the corresponding interacting particle systems. We should emphasize that the method used in the proof of uniform in $N$ convergence in the sense of relative entropy for the corresponding interacting particle systems are completely different from that in \cite{GLWZ,GM,CLRW}, whose approach heavily relies on the explicit representation of the invariant probability measure and on all the tools and criteria for log-Sobolev inequalities in this situation, see the review of the recent literature in~\cite{M26}. Instead, we first study the hypercontractivity for a referenced SDE and the hypercontractivity preserves for $N$-independent referenced SDEs. When the interaction is weak enough, we adopt the Girsanov's transform to derive the uniform in $N$ hypercontractivity, which together with the $L^1$-gradient estimate implies the uniform in $N$ log-Sobolev inequality for the original interacting particle system.

\bigskip

The remainder of this work is organized as follows. In Section \ref{sec: main} we present our main results, Theorems \ref{hyc},  \ref{cty} and \ref{cty12pn}. More specifically, Theorem \ref{hyc} establishes exponential ergodicity in both relative entropy and the $L^2$-Wasserstein distance for a broad class of kinetic SDEs under a partially dissipative condition. Theorems \ref{cty} and \ref{cty12pn} extend these results to the McKean –Vlasov setting and to the associated mean-field interacting particle systems, with convergence rates that are uniform in the number of particles in the latter case.  Section \ref{sec: proofs} is devoted to the proofs of these results.

\section{Main results}\label{sec: main}

Let $\scr P(\R^{2d})$ be all the probability measures on $\R^{2d}$ equipped with the weak topology. For $p\geq 1$, let
$$\scr P_p(\R^{2d}):=\big\{\mu\in \scr P(\R^{2d}): \|\mu\|_p:=\mu(|\cdot|^p)^{\ff 1 p}<\infty\big\},$$
which is a Polish space under the  $L^p$-Wasserstein distance
$$\W_p(\mu,\nu)= \inf_{\pi\in \C(\mu,\nu)} \bigg(\int_{\R^{2d}\times\R^{2d}} |x-y|^p \pi(\d x,\d y)\bigg)^{\ff 1 {p}},\ \ \mu,\nu\in\scr P_p(\R^{2d}), $$ where $\C(\mu,\nu)$ is the set of all couplings of $\mu$ and $\nu$.

For $\nu,\mu \in \scr P(\R^{2d})$, we write $\mathrm{Ent}(\nu|\mu)$ the relative entropy of $\nu$ with respect to $\mu$. The law of a random variable $Z$ is denoted by $\L_{Z}$.

\subsection{Exponential Ergodicity for classical kinetic SDEs}
Our first main result provides  exponential ergodicity for classical kinetic SDEs.
 Let $\sigma\in\R^d\otimes\R^n$, $b:\R^{2d}\to\R^d$ be measurable and locally bounded and $\{W_t\}_{t\geq 0}$ be an $n$-dimensional standard Brownian motion on some complete filtered probability space $(\Omega, \scr F, (\scr F_t)_{t\geq 0},\P)$.

 We make the following assumption.
\begin{enumerate}
\item[{\bf(B)}] There exists a constant $K_b>0$ such that for all $z,\bar{z}\in\R^{2d}$,
\begin{align}\label{Lipsc1}|b(z)-b(\bar{z})| \leq K_b|z-\bar{z}|.
\end{align}
Moreover, there exist constants $\theta,r,R> 0,$ and $r_0\in(-1,1)$ such that for any $x,y,\bar{x},\bar{y}\in\R^{d}$ satisfying $|x-\bar{x}|^2+|y-\bar{y}|^2\geq R^2$,
\begin{align}\label{Patdi1}
\nonumber&\<r^2(x-\bar{x})+rr_0(y-\bar{y}),y-\bar{y}\>\\
&+\<(y-\bar{y})+rr_0(x-\bar{x}),b(x,y)-b(\bar{x},\bar{y})\>\\
\nonumber&\leq -\theta(|x-\bar{x}|^2+|y-\bar{y}|^2).
\end{align}
\end{enumerate}
\begin{rem}
 Condition {\bf(B)} is equivalent to $\cite[{\bf Assumption\ \ 1}]{M25}$, which can be seen by noting that $r_0\in(-1,1)$ and considering the matrix
 \[A=\left(
         \begin{array}{cc}
           r^2I_{d\times d} & rr_0I_{d\times d} \\
           rr_0I_{d\times d} & I_{d\times d} \\
         \end{array}
       \right)\,.\]
\end{rem}
%Let $$\psi((x,y),(\bar{x},\bar{y}))=\sqrt{r^2|x-\bar{x}|^2/2+|y-\bar{y}|^2/2+rr_0\<x-\bar{x},y-\bar{y}\>},\ \ (x,y),(\bar{x},\bar{y})\in\R^{2d}.$$
%Then there exists a constant $C_\psi>1$ depending on $r,r_0$ such that for any $(x,y),(\bar{x},\bar{y})\in\R^{2d}$,
%\begin{align}\label{cpsid}C_\psi^{-1}(|x-\bar{x}|^2+|y-\bar{y}|^2)\leq \psi((x,y),(\bar{x},\bar{y}))^2\leq C_\psi(|x-\bar{x}|^2+|y-\bar{y}|^2).
%\end{align}

\begin{exa}\label{EXa10}
As detailed in \cite[Example 1]{M25}, for any $\beta>0$, there exists $\kappa>0$ such that {\bf(B)} is satisfied for $b(x,y) = -x-\beta y + F(x,y)$ provided $|\nabla F| \leqslant \kappa$ outside a compact set. Notice that, even in the equilibrium case $b(x,y)=-\nabla V(x)-\beta y$, as shown in~\cite[Proposition 4]{M23}, even if $V$ is strongly convex with unbounded Hessian, as soon as it is not exactly quadratic then, for small values of $\beta$,~\eqref{Patdi1}  is \emph{not} satisfied for any $\theta,r,R,r_0$.
\end{exa}

  We consider the kinetic SDE
\begin{align}\label{Degth1}\left\{
  \begin{array}{ll}
    \d X_t=Y_t\d t,  \\
    \d Y_t=b(X_t,Y_t)\d t+\sigma\d W_t.
  \end{array}
\right.
\end{align}

Note that \eqref{Lipsc1} implies that \eqref{Degth1} is well-posed. We denote by $(X_t^z, Y_t^z)$ the solution to \eqref{Degth1} with initial value $z\in\R^{2d}$ and by $(X_t^{\mu_0},Y_t^{\mu_0})$  the solution starting from $\mu_0\in\scr P$. We further denote by $P_t$ the associated semigroup.

\begin{thm}\label{hyc}Assume  {\bf(B)} and $\sigma\sigma^\ast$ is invertible. Then \eqref{Degth1} has a unique invariant probability measure $\bar{\mu}\in\scr P_2(\R^{2d})$ and there exists a constant $t_1>0$ such that $P_{t_1}^\dagger $, the adjoint operator of $P_{t_1}$ in $L^2(\bar{\mu})$, is hypercontractive, in the sense that
\begin{align}\label{hyper5}
\|P_{t_1}^\dagger\|_{L^{\frac{4}{3}}(\bar{\mu})\to L^{2}(\bar{\mu})}= 1.
\end{align}
Consequently, there exist constants $c,\tilde c,\lambda>0$ such that
\begin{align}\label{entc1} \mathrm{Ent}(\L_{(X_t^{\mu_0},Y_t^{\mu_0})}|\bar{\mu})\leq c\e^{-2\lambda t}\mathrm{Ent}(\mu_0|\bar{\mu}),~ t\geq 0, \mu_0\in\scr P(\R^{2d}),
\end{align}
and
\begin{align}\label{w2con1}\W_2(\L_{(X_t^{\mu_0},Y_t^{\mu_0})},\bar{\mu})\leq \tilde c\e^{-\lambda t}\W_2(\mu_0,\bar{\mu}),\ \ t\geq 0, \mu_0\in\scr P_2(\R^{2d}).
\end{align}
\end{thm}

\begin{rem}
In the case $n=d$, $\sigma=\sigma_0I_{d\times d}$ for some positive constant $\sigma_0$, $L^2$-exponential ergodicity for \eqref{Degth1} has been derived in \cite{M25}. In fact, the result in \cite{M25} can be easily extended to the general case of $\sigma\sigma^\ast$ being an invertible $d\times d$ matrix, as we shall see in  Lemma \ref{ExpL2} below.
\end{rem}
\begin{rem}
In the uniform dissipative case, \cite{FYW2017} uses hypercontractivity to derive the exponential ergodicity in relative entropy. Let us  remark that in the uniform dissipative case, the hypercontractivity can be derived by Wang's Haranck inequality with power while in the partially dissipative case, Wang's Harnack inequality can only derive hyperboundedness, see \cite{HKR25}. Recently, \cite{M25} derives the $L^2$-exponential ergodicity under the partially dissipative condition. Note that exponential ergodicity in relative entropy is strictly stronger than $L^2$-exponential ergodicity, see for instance \cite[Proposition 2.3]{FYW2017}.
Together with hyperboundedness, $L^2$-exponential ergodicity implies exponential ergodicity in relative entropy, as we shall recall in Lemma \ref{Hypercon} below.
\end{rem}

\subsection{Exponential ergodicity for McKean-Vlasov SDEs}

Our second main result generalizes exponential ergodicity towards distribution dependent SDEs.
Let $b:\R^{2d}\times \scr P_2(\R^{2d})\to\R^d$, $\sigma:\scr P_2(\R^{2d})\to\R^d\otimes\R^n$ be measurable and bounded on bounded sets. We assume the following.
\begin{enumerate}
\item[{\bf(C)}] There exist constants $K_b>0,K_I>0$ such that for all $z,\bar{z}\in\R^{2d}, \gamma,\tilde{\gamma}\in\scr P_2(\R^{2d})$,
\begin{align}\label{Lipsc}|b(z,\gamma)-b(\bar{z},\tilde{\gamma})|+\|\sigma(\gamma)-\sigma(\tilde{\gamma})\|_{HS} \leq K_b|z-\bar{z}|+K_I\W_2(\gamma,\tilde{\gamma}).
\end{align}
Moreover, there exist constants $\theta,r,R>0, $ and $r_0\in(-1,1)$ such that for any $x,y,\bar{x},\bar{y}\in\R^{d}$ satisfying $|x-\bar{x}|^2+|y-\bar{y}|^2\geq R^2$ and $\mu\in\scr P_2(\R^{2d})$,
\begin{align}\label{Patdi}
\nonumber&\<r^2(x-\bar{x})+rr_0(y-\bar{y}),y-\bar{y}\>\\
&+\<(y-\bar{y})+rr_0(x-\bar{x}),b(x,y,\mu)-b(\bar{x},\bar{y},\mu)\>\\
\nonumber&\leq -\theta(|x-\bar{x}|^2+|y-\bar{y}|^2).
\end{align}
In addition, there exist constants $0<\delta_2\leq\delta_1$ such that
\begin{align}\label{eup}
\delta_2\leq \sigma\sigma^\ast\leq \delta_1.
\end{align}
\end{enumerate}

We consider the McKean-Vlasov SDE
\begin{align}\label{Degth}\left\{
  \begin{array}{ll}
    \d X_t=Y_t\d t,  \\
    \d Y_t=b(X_t,Y_t,\L_{(X_t,Y_t)})\d t+\sigma(\L_{(X_t,Y_t)})\d W_t.
  \end{array}
\right.
\end{align}
First we note that under \eqref{Lipsc}, \eqref{Degth} is well-posed in $\scr P_2(\R^{2d})$. We use $P_t^\ast \mu_0$ to denote the distribution of the solution to \eqref{Degth} with initial distribution $\mu_0\in\scr P_2(\R^{2d})$.

\begin{thm}\label{cty}Assume {\bf (C)}. There exist  $K_*,c,\tilde c,\lambda>0$, depending on $d$ and on the constants in {\bf (C)} except on $K_I$, such that, provided $K_I\leqslant K_*$, then  \eqref{Degth} has a unique stationary solution $\bar{\mu}\in\scr P_2(\R^{2d})$,
\begin{align}\label{w2con}\W_2(P_t^\ast\mu_0,\bar{\mu})\leq c\e^{-\lambda t}\W_2(\mu_0,\bar{\mu}),\ \ t\geq 0, \mu_0\in\scr P_2(\R^{2d}),
\end{align}
and
%$$\W_2(P_t^\ast\mu_0,\bar{\mu})\leq c\e^{-\lambda t}\W_2(\mu_0,\bar{\mu}),\ \ t\geq 0; \mathrm{Ent}(P_t^\ast\mu_0|\bar{\mu})\leq c^2\e^{-2\lambda t}\mathrm{Ent}(\mu_0|\bar{\mu}),\ \ t\geq 1.$$
%and
\begin{align}\label{entco} \mathrm{Ent}(P_t^\ast\mu_0|\bar{\mu})\leq \tilde{c}\e^{-2\lambda t}\mathrm{Ent}(\mu_0|\bar{\mu}),~ t\geq 1, \mu_0\in\scr P_2(\R^{2d}).
\end{align}
\end{thm}
\begin{rem}
In \cite[Theorem 3.9]{HKR25}, the authors derived the exponential ergodicity in $L^2$-Wasserstein distance and relative entropy for the case that $b(x,y,\mu)=-y-\nabla V(x)-\int_{\R^{2d}}\nabla_x W(x,z)\mu(\d z)$ with $x,y\in\R^d$.  In the present Theorem \ref{cty} we extended this result to more general drifts $b$, which need not be of gradient form. This extension is obtained by combining the $L^2$-exponential ergodicity established in \cite[Theorem 2]{M25} with the hyperboundedness and $log$-Soblev inequlity proved in \cite[Theorem 3.4, Theorem 3.5]{HKR25} for classical kinetic SDEs.
\end{rem}
\subsection{Uniform in $N$ exponential ergodicity for mean field interacting particle system}
Our third main result provides a uniform in $N$ exponential ergodicity for mean field interacting kinetic particle systems of the form
\begin{align}\label{Degthp}\left\{
  \begin{array}{ll}
    \d X_t^{i,N}=Y_t^{i,N}\d t,  \\
    \d Y_t^{i,N}=b\left(X_t^{i,N},Y_t^{i,N},\frac{1}{N}\sum_{i=1}^N\delta_{(X_t^{i,N},Y_t^{i,N})}\right)\d t+\sigma\d W_t^i,
  \end{array}
\right.
\end{align}
where $(W_t^i)_{i\geq 1}$ are independent $n$-dimensional Brownian motions.
In this case, we are not able to consider the case where $\sigma$ depends on the empirical measure of the particles, since the Harnack inequality for multiplicative kinetic SDEs is still open.
The corresponding McKean-Vlasov SDE is
\begin{align}\label{Degth3}\left\{
  \begin{array}{ll}
    \d X_t=Y_t\d t,  \\
    \d Y_t=b(X_t,Y_t,\L_{(X_t,Y_t)})\d t+\sigma\d W_t.
  \end{array}
\right.
\end{align}

%This implies that under \eqref{Lipsc}, \eqref{Degthp} is well-posed and we denote by $P^N_t$ the associated semigroup.
 Let us denote by $(P_t^N)^\ast\nu^N$ the distribution of the solution to \eqref{Degthp} with initial distribution $\nu^N\in\scr P((\R^{2d})^N)$. Let $P_t^N$ be the associated semigroup to \eqref{Degthp}.

\begin{thm}\label{cty12pn} Assume {\bf (C)}.
Then there exist  $K_*,c,\lambda,T,C_0>0$, depending on $d$ and on the constants in {\bf (C)} except $K_I$, such that the following assertions hold. Provided $K_I\leqslant K_*$  then, for all $N\geqslant 1$,  \eqref{Degthp} has a unique invariant probability measure $\mu^N\in\scr P_2((\R^{2d})^N)$ satisfying
$$\|P_T^N\|_{L^2(\mu^N)\to L^4(\mu^N)}=1$$
and the uniform in $N$ log-Sobolev inequality
\begin{align*}&\mu^N(f^2 \log f^2)\leq  C_0\mu^N(|\nabla f|^2),\ \ \mu^N(f^2)=1,
\end{align*}
from which
\begin{align}\label{w2conpn}\W_2((P_t^N)^\ast\nu^N,\mu^N)^2\leq c\e^{-2 \lambda t}\W_2(\nu^N,\mu^N)^2,\ \ t\geq 0, \nu^N\in\scr P_2((\R^{2d})^N),
\end{align}
and
%$$\W_2(P_t^\ast\mu_0,\bar{\mu})\leq c\e^{-\lambda t}\W_2(\mu_0,\bar{\mu}),\ \ t\geq 0; \mathrm{Ent}(P_t^\ast\mu_0|\bar{\mu})\leq c^2\e^{-2\lambda t}\mathrm{Ent}(\mu_0|\bar{\mu}),\ \ t\geq 1.$$
%and
\begin{align}\label{entcopn} \mathrm{Ent}((P_t^N)^\ast\nu^N|\mu^N)\leq c\e^{-2\lambda t}\mathrm{Ent}(\nu^N|\mu^N),~ t\geq 0, \nu^N\in\scr P_2((\R^{2d})^N).
\end{align}
\end{thm}

\begin{rem}
%Compared with the existing results on the exponential ergodicity in $L^2$-Wasserstein distance and relative entropy for second-order McKean-Vlasov SDEs/mean field interacting particle system uniform in the number of particles under partially dissipative condition (such as \cite{GM,CLRW} ; see also \cite{GLWZ} for the case with non-degenerate additive noise), the results in Theorems~\ref{cty} and~\ref{cty12pn} are the first ones where the drift term is not of gradient form and the invariant probability measure for the associated mean field interacting particle system has no explicit density. \textcolor{red}{Moreover, as illustrated in Introduction, the method adopted in the proof of the uniform in $N$ log-Sobolev inequality for mean field interacting particle system is, to our knowledge,  original and quite different from that in \cite{GLWZ} and it applies to general $b$. \\
Compared with the existing results on the exponential ergodicity in $L^2$-Wasserstein distance and relative entropy for second-order McKean-Vlasov SDEs/mean field interacting particle system under partially dissipative conditions, with particle system estimates uniform in the number of particles (such as \cite{GM,CLRW} ; see also \cite{GLWZ} for the case with non-degenerate additive noise),  Theorems~\ref{cty} and~\ref{cty12pn} allow general drifts that are not required to be of gradient form and do not require explicit expression for the invariant probability measure of the associated particle system.   To the best of our knowledge,  these results overcome a long-standing difficulty in establishing exponential ergodicity in this non-equilibrium setting.  The novelty in Theorem \ref{cty12pn} not only in the results but also in the method of proof.  In particular,  as explained in the introduction, we introduce a new approach to establishing the $log$-Sobolev inequality for the mean field interacting particle system uniform in $N$.  This approach is different from that in \cite{GLWZ} and applies to general drift $b$,  without dependence on gradient structure or an explicit representation of the invariant measure.
\end{rem}

%\begin{rem}
%When $K_I$ is small enough, \eqref{Degth3}  admits a unique stationary probability distribution $\bar{\mu}$ thanks to Theorem~\ref{cty}. Along the proof of Theorem~\ref{cty12pn}, we establish that $ \W_2(\bar{\mu}^{\otimes N},\bar{\mu}^N )^2 \leqslant C N R_N$ for some constant $C>0$. Using triangular inequality,  this shows that \eqref{w2conpn} also holds if we replace $\bar{\mu}^N$ by $\bar{\mu}^{\otimes N}$. Then, if particles are indistinguishable (i.e. if $\nu^N$ is invariant by permutation of the particles' labels), writing $\nu_{t}^{k,N}$ the marginal law of $k \leqslant N$ particles under the law $(P_t^N)^\ast\nu^N$ and using the scaling properties of $\W_2$, we get
%\[\W_2(\nu_{t}^{k,N},\bar{\mu}^{\otimes k})^2\leq c\e^{-2 \lambda t}\W_2(\nu_0^{k,N},\bar{\mu}^{\otimes k})^2+ckR_{d}(N),\ \ t\geq 0.\]
%\end{rem}

\begin{exa}
As in Example \ref{EXa10}, for any $\beta>0$, there exists $\kappa>0$ such that the conditions in Theorem \ref{cty} and Theorem \ref{cty12pn} are satisfied for $$b(z,\mu) = -x-
\beta y +F_1(x,y)+\int_{\R^{2d}}F_2(z,\bar{z})\mu(\d \bar{z}),\ \ z=(x,y)\in\R^{2d}$$ with $F_1,F_2:\R^{2d}\times\R^{2d}\to\R^{d}$ provided $|\nabla F_1| \leqslant \kappa$, $|\nabla_1 F_2| \leqslant \kappa$ outside a compact set and $\|\nabla_2 F_2\|_{\infty}$ is small enough.
\end{exa}

\section{Proof of the Main Theorems}\label{sec: proofs}
Before we go on and prove our main results, we shall mention the following key result for our analysis.
%\subsection{Exponential ergodicity in relative entropy from $L^2$-exponential ergodicity and hyperboundedness}
In order to do so, consider a conservative diffusion process on $\R^m$ with generator
$L=\frac{1}{2}\mathrm{tr}(a\nabla^2)+\<b,\nabla\>,$
and associated semigroup $P_t$. Let $\mu$ be an invariant probability measure of $(P_t)_{t\ge0}$. By \cite[Proposition 2.3]{FYW2017}, the exponential convergence for $\mu(P_tf \log P_t f)$ can be derived from hypercontractivity of $P_t$, i.e. $\|P_{t_0}\|_{L^2(\mu)\to L^4(\mu)}\leq 1$ for some $t_0>0$. Moreover, in that case, the entropy convergence rate depends only on $t_0$. On the other hand, \cite[Step (d) in the proof of Theorem 1.1]{BWY15} states that hypercontractivity of $P_t$ can be deduced by the exponential convergence of $\mu((P_tf)^2)$ together with hyperboundedness $\|P_{\tilde{t}_0}\|_{L^2(\mu)\to L^4(\mu)}\leq C$ for some constants $\tilde{t}_0>0, C>0$. This argument can be traced back to \cite{Glimm} for the derivation of the contractivity for the Mehler semigroup. Building on \cite[Proposition 2.3]{FYW2017} and \cite[Step (d) in Proof of Theorem 1.1]{BWY15}, we summarize  the following key result on   exponential ergodicity in relative entropy.

\begin{lem}\label{Hypercon} Let $\bar{\mu}$ be the invariant probability measure of $(P_t)_{t\ge0}$, and
assume that $(P_t)_{t\ge0}$ is hyperbounded, i.e.,
$\|P_{t_0}\|_{p\to q }:= \|P_{t_0}\|_{L^p(\bar{\mu})\to L^q(\bar{\mu})}<\infty $ for some   $t_0>0 $ and  $1<p<q<\infty$.
Assume moreover that $L^2$-exponential ergodicity holds, i.e.
\begin{align*}\|P_t f-\bar{\mu}(f)\|_{L^2(\bar{\mu})}\leq c\e^{-\lambda t}\|f-\bar{\mu}(f)\|_{L^2(\bar{\mu})},\ \ t\geq 0
\end{align*}
for some constants $c,\lambda>0$.
Then hypercontractivity holds, i.e.
\begin{align*}
\|P_{\tilde{t}_0}\|_{p\to q }\leq 1
\end{align*}
for some constant $\tilde{t}_0>0$ depending only on $c,\lambda,q,p,t_0$ and $\|P_{t_0}\|_{p\to q }$.
Consequently, exponential ergodicity in relative entropy holds:
\begin{align*}
\bar{\mu}(P_t f\log P_tf)\leq \tilde{c}\e^{-\tilde{\lambda} t}\bar{\mu}(f\log f),\ \ \bar{\mu}(f)=1, f>0, t\geq 0
\end{align*}
for some constants $\tilde{c}, \tilde{\lambda}>0$ depending only on $\tilde{t_0},q,p$.
\end{lem}
\begin{rem} An introduction to hypercontractivity can be found in \cite{BE}.
\end{rem}

\subsection{Proof of Theorem \ref{hyc}}
In this section, we prove exponential ergodicity in relative entropy and $L^2$-Wasserstein distance for classical kinetic SDEs.
%To this end, we make the following assumption.
%\begin{enumerate}
%\item[{\bf(B)}] There exists a constant $K_b>0$ such that for all $z,\bar{z}\in\R^{2d}$,
%\begin{align}\label{Lipsc1}|b(z)-b(\bar{z})| \leq K_b|z-\bar{z}|.
%\end{align}
%Moreover, there exist constants $\theta,r,R> 0,$ and $r_0\in(-1,1)$ such that for any $x,y,\bar{x},\bar{y}\in\R^{d}$ satisfying $|x-\bar{x}|^2+|y-\bar{y}|^2\geq R^2$,
%\begin{align}\label{Patdi1}
%\nonumber&\<r^2(x-\bar{x})+rr_0(y-\bar{y}),y-\bar{y}\>\\
%&+\<(y-\bar{y})+rr_0(x-\bar{x}),b(x,y)-b(\bar{x},\bar{y})\>\\
%\nonumber&\leq -\theta(|x-\bar{x}|^2+|y-\bar{y}|^2).
%\end{align}
%\end{enumerate}

We first present a lemma on the $L^2$-exponential ergodicity for \eqref{Degth1}. The proof largely follows the approach in \cite[Theorem 2]{M25}, where the technique of hypocoercivity plays a central role.
\begin{lem}\label{ExpL2} Assume  \eqref{Lipsc1} and $\sigma\sigma^\ast$ is invertible. If in addition \eqref{Degth1} has a unique invariant probability measure $\bar{\mu}$ such that the Poincar\'{e} inequality
$$\bar{\mu}(f^2)-\bar{\mu}(f)^2\leq C_{\mathrm{PI}}\bar{\mu}(|\nabla f|^2),\ \ f\in C_0^\infty(\R^{2d})$$
holds for some constant $C_{\mathrm{PI}}$, then there exist constants $c_0>0,\lambda_0>0$ depending on $C_{PI}$, $K_b$ and the minimum eigenvalue of $\sigma\sigma^\ast$ such that
\begin{align*}\|P_t f-\bar{\mu}(f)\|_{L^2(\bar{\mu})}\leq c_0\e^{-\lambda_0 t}\|f-\bar{\mu}(f)\|_{L^2(\bar{\mu})},\ \ t\geq 0,\ f\in L^2(\mu).
\end{align*}
\end{lem}
\begin{proof}
%Let $$M=2(\|\nabla_x b\|_\infty+\|\nabla_y b\|_\infty+1)^2+(\|\nabla_x b\|_\infty+\|\nabla_y b\|_\infty+1)$$
Let $$M=2(2K_b+1)^2+(2K_b+1).$$
Let $\delta_1$ be the minimum eigenvalue of $\sigma\sigma^\ast$ and $\varepsilon=\frac{\delta_1}{M+\frac{1}{2}}$.
By an approximation technique, we can assume that $b\in C^1$ with $\|\nabla b\|_\infty<K_b$.
For $f\in C_0^\infty(\R^{2d})$, let $f_t=P_t f-\bar{\mu}(f)$. Then it follows from the fact
\begin{align}\label{akt}L (g^2)=2g Lg+|\sigma^\ast\nabla_y g|^2,\ \ \bar{\mu}(L (g^2))=0,\ \ g\in C_0^\infty(\R^{2d})
\end{align}
that
\begin{align*}
\frac{\d\|f_t\|_{L^2(\bar{\mu})}^2}{\d t}=2\bar{\mu}(f_tL f_t) =-\bar{\mu}(|\sigma^\ast\nabla_y f_t|^2).
\end{align*}
Let $\alpha(t)=1-\e^{-t/3}$,
$$G_t=\varepsilon\left(
        \begin{array}{cc}
          \alpha^3(t)I_{d\times d} & -\alpha^2(t)I_{d\times d} \\
          -\alpha^2(t)I_{d\times d} & \alpha(t)I_{d\times d} \\
        \end{array}
      \right),
$$
and denote
$$J=\left(
  \begin{array}{cc}
   0_{d\times d} & \nabla_x b \\
      I_{d\times d} & \nabla_yb \\
  \end{array}
\right).
$$
Then it is easy to see that
$$\frac{\d G_t}{\d t}=\varepsilon\left(
        \begin{array}{cc}
          3\alpha^2(t)\alpha'(t)I_{d\times d} & -2\alpha(t)\alpha'(t)I_{d\times d} \\
          -2\alpha(t)\alpha'(t)I_{d\times d} & \alpha'(t)I_{d\times d}\\
        \end{array}
      \right)
$$
and
\begin{align*}2G_tJ&=2\varepsilon\left(
        \begin{array}{cc}
          \alpha^3(t) I_{d\times d}& -\alpha^2(t)I_{d\times d} \\
          -\alpha^2(t)I_{d\times d} & \alpha(t)I_{d\times d} \\
        \end{array}
      \right)\left(
  \begin{array}{cc}
   0_{d\times d} & \nabla_x b \\
      I_{d\times d} & \nabla_yb \\
  \end{array}
\right)\\
&=2\varepsilon\left(
  \begin{array}{cc}
   -\alpha^2(t)I_{d\times d} & \alpha^3(t)\nabla_x b-\alpha^2(t)\nabla_yb \\
      \alpha(t)I_{d\times d} & -\alpha^2(t)\nabla_x b+\alpha(t)\nabla_yb \\
  \end{array}
\right).
\end{align*}
Moreover, for $z=(x,y)\in\R^{2d}$, it holds
$$\<G_t z,z\>=\vv\alpha(t)(\alpha^2(t)|x|^2-2\alpha(t)\<x,y\>+|y|^2)\leq 2\varepsilon|z|^2. $$
Observing
\begin{align*}
\nabla Lf_t=L\nabla f_t+J\nabla f_t,
\end{align*}
and in view of \eqref{akt},
$$\bar{\mu}(\<G_tL\nabla f_t,\nabla f_t\>)=\bar{\mu}(\<L \sqrt{G_t}\nabla f_t,\sqrt{G_t}\nabla f_t\>)=-\bar{\mu}(|\sigma^\ast \nabla_y\sqrt{G_t} \nabla f_t|^2)\leq 0,$$
we conclude that
\begin{align*}
&\frac{\d \bar{\mu}(\<G_t\nabla f_t,\nabla f_t\>)}{\d t}\\
&=\bar{\mu}(\<\frac{\d G_t}{\d t}\nabla f_t,\nabla f_t\>)+2\bar{\mu}(\<G_t\nabla Lf_t,\nabla f_t\>)\\
&=\bar{\mu}(\<\frac{\d G_t}{\d t}\nabla f_t,\nabla f_t\>)+2\bar{\mu}(\<G_tL\nabla f_t,\nabla f_t\>)+2\bar{\mu}(\<G_tJ\nabla f_t,\nabla f_t\>)\\
&\leq \bar{\mu}(\<\frac{\d G_t}{\d t}\nabla f_t,\nabla f_t\>)+2\bar{\mu}(\<G_tJ\nabla f_t,\nabla f_t\>).
\end{align*}
Set
$$\scr N_t=\|f_t\|_{L^2(\bar{\mu})}^2+\bar{\mu}(\<G_t\nabla f_t,\nabla f_t\>).$$
Then we have
$$\frac{\d \scr N_t}{\d t}\leq \bar{\mu}\left(R_t\nabla f_t,\nabla f_t\right)$$
where
 \begin{align*}
R_t&=-\left(
                                                 \begin{array}{cc}
                                                   0_{d\times d} & 0_{d\times d} \\
                                            0_{d\times d} & \sigma\sigma^\ast
                                                 \end{array}
                                               \right)
+\frac{\d G_t}{\d t}+2G_tJ\\
&=\left(
                                                 \begin{array}{cc}
                                                  \varepsilon(3\alpha'(t)-2)\alpha^2(t)I_{d\times d} & \vv(2\alpha^3(t)\nabla_x b-2\alpha^2(t)\nabla_yb-2\alpha(t)\alpha'(t)) \\
                                            \vv 2\alpha(t)(1-\alpha'(t)) & -\sigma\sigma^\ast+\vv\alpha'(t)-\vv\alpha^2(t)\nabla_x b+\vv\alpha(t)\nabla_yb
                                                 \end{array}
                                               \right).
\end{align*}
 It is not difficult to see from $\alpha'(t)\in(0,1/3]$ and $\alpha\in[0,1)$ that
\begin{align*}
\<R_t z,z\>&= \varepsilon(3\alpha'(t)-2)\alpha^2(t)|x|^2+\<[\vv(2\alpha^3(t)\nabla_x b-2\alpha^2(t)\nabla_yb-2\alpha(t)\alpha'(t))]y,x\>\\
&+\<[\vv 2\alpha(t)(1-\alpha'(t))x, y\>+\<[-\sigma\sigma^\ast+\vv\alpha'(t)-\vv\alpha^2(t)\nabla_x b+\vv\alpha(t)\nabla_yb]y,y\>\\
&\leq -\vv\alpha^2(t)|x|^2-\<\sigma\sigma^\ast y,y\>\\
&+\vv2\alpha(t)|x|(\|\nabla_x b\|_\infty+\|\nabla_y b\|_\infty+1)|y|+\vv(\|\nabla_x b\|_\infty+\|\nabla_y b\|_\infty+1)|y|^2\\
&\leq -\frac{\vv\alpha^2(t)}{2}|x|^2-\delta_1|y|^2\\
&+\vv\{2(\|\nabla_x b\|_\infty+\|\nabla_y b\|_\infty+1)^2+(\|\nabla_x b\|_\infty+\|\nabla_y b\|_\infty+1)\}|y|^2\\
&\leq -\frac{\vv\alpha^2(t)}{2}|x|^2+(\vv M-\delta_1)|y|^2.
\end{align*}
Therefore, we get
$$\frac{\d \scr N_t}{\d t}\leq -\frac{\vv\alpha^2(t)}{2}\bar{\mu}(|\nabla f_t|^2).$$
Observe that it follows from Poincar\'{e}'s inequality that
$$\bar{\mu}(|\nabla f_t|^2)\geq \frac{1}{C_{P_I}}\bar{\mu}(|f_t|^2)=\frac{1}{C_{P_I}}[\scr N_t-\bar{\mu}(\<G_t\nabla f_t,\nabla f_t\>)]\geq \frac{1}{C_{P_I}}[\scr N_t-2\varepsilon\bar{\mu}(|\nabla f_t|^2)],$$
which implies
$$\bar{\mu}(|\nabla f_t|^2)\geq \frac{1}{C_{P_I}+2\varepsilon}\scr N_t.$$
Consequently, we derive
$$\frac{\d \scr N_t}{\d t}\leq -\frac{\vv\alpha^2(t)}{2C_{P_I}+4\varepsilon}\scr N_t.$$
Gronwall's inequality and $\scr N_0=\|f_0\|_{L^2(\bar{\mu})}^2$ yields
$$\|f_t\|_{L^2(\bar{\mu})}^2\leq\scr N_t\leq \exp\left\{-\frac{\vv}{2C_{P_I}+4\varepsilon}\int_0^t\alpha(s)^2\d s\right\}\|f_0\|_{L^2(\bar{\mu})}^2.$$
The proof is completed by noting that for $t\geq 1$, $$\int_0^t\alpha^2(s)\d s\geq\int_1^t(1-\e^{-1/3})^2\d s=(1-\e^{-1/3})^2(t-1).$$
\end{proof}

\begin{proof}[Proof of Theorem \ref{hyc}]
(Step 1) The existence and uniqueness of an invariant probability measure is known, see for instance \cite[Proof of Theorem 3.3]{HKR25}. By \cite[Theorem 3.4]{HKR25}, there exist constants $t_0>0$, $C_0>0$ such that
\begin{align}\label{hyper31}
\|P_{t_0}\|_{L^2(\bar{\mu})\to L^{4}(\bar{\mu})}\leq C_0.
\end{align}
Moreover, it follows from \cite[Theorem 3.5]{HKR25} that the log-Sobolev inequality holds:
\begin{align}\label{logsob1}\bar{\mu}(f\log f)\leq C_{1}\bar{\mu}(|\nabla \sqrt{f}|^2),\ \ f\in C_0^\infty(\R^{2d}), f>0,\bar{\mu}(f)=1,
\end{align}
for some $C_{1}>0$. This immediately implies the Poincar\'{e} inequality,
$$\bar{\mu}(f^2)-\bar{\mu}(f)^2\leq C_{1}\bar{\mu}(|\nabla f|^2),\ \ f\in C_0^\infty(\R^{2d}).$$
 According to Lemma \ref{ExpL2}, the $L^2$-exponential ergodicity  holds, i.e. there exist $c_0,\lambda_0>0$ such that
\begin{align}\label{coL}\|P_t f-\bar{\mu}(f)\|_{L^2(\bar{\mu})}\leq c_0\e^{-\lambda_0 t}\|f-\bar{\mu}(f)\|_{L^2(\bar{\mu})},\ \ t\geq 0,\ f\in L^2(\mu).
\end{align}
Hence, \eqref{coL} together with \eqref{hyper31} and Lemma \ref{Hypercon} implies hypercontractivity for large enough $t_1\geq t_0$,
\begin{align}\label{hyperne}
\|P_{t_1}\|_{L^2(\bar{\mu})\to L^{4}(\bar{\mu})}= 1.
\end{align}

(Step 2) Next, we prove the hypercontractivity of $P_{t_1}^\dagger$. For any $f\in L^{\frac{4}{3}}(\bar{\mu})$, set
$$f_n=(-n\vee f)\wedge n,\ \ n\geq 1.$$
Then $f_n\in L^2(\bar{\mu})$, which together with the fact that $P_{t_1}^\dagger$ is the adjoint operator of $P_t$ in $L^2(\bar{\mu})$ implies that
\begin{align*}\|P_{t_1}^\dagger |f_n|\|_{L^2(\bar{\mu})}=\sup_{g\in L^2(\bar{\mu})}\int_{\R^{2d}}(P_{t_1}^\dagger |f_n|) g\d \bar{\mu}&=\sup_{g\in L^2(\bar{\mu})}\int_{\R^{2d}}|f_n| P_{t_1}g\d \bar{\mu}\\
&\leq \sup_{g\in L^2(\bar{\mu})}\|P_{t_1}g\|_{L^4 (\bar{\mu})}\|f_n\|_{L^{\frac{4}{3}}(\bar{\mu})}\\
&\leq \|P_{t_1}\|_{L^2(\bar{\mu})\to L^{4}(\bar{\mu})}\|f_n\|_{L^{\frac{4}{3}}(\bar{\mu})}=\|f_n\|_{L^{\frac{4}{3}}(\bar{\mu})}.
\end{align*}
By Fatou's lemma and dominated convergence theorem, we derive \eqref{hyper5}, which combined with Lemma \ref{Hypercon} implies \eqref{entc1}.

(Step 3) By \eqref{Lipsc1}, it is standard to derive
\begin{align}\label{w2sto}\sup_{t\in[0,1]}\W_2(\L_{(X_t^{\mu_0},Y_t^{\mu_0})},\bar{\mu})\leq c_0\W_2(\mu_0,\bar{\mu}),\ \  \mu_0\in\scr P_2(\R^{2d}).
\end{align}
for some constant $c_0>0$.
By \cite{RW}, a log-Harnack inequality holds:
\begin{align*}
P_1\log f(x)\leq \log P_1 f(y)+c_1|x-y|^2,
\end{align*}
for some $c_1>0$. By \cite[Proposition 1.4.4(3)]{FYWang} this implies
%This implies
\begin{align}\label{enc}
\mathrm{Ent}(\L_{(X_1^{\mu_0},Y_1^{\mu_0})},\L_{(X_1^{\nu_0},Y_1^{\nu_0})})\leq c_1\W_2(\mu_0,\nu_0)^2.
\end{align}
 Moreover, by \cite{OV},~\eqref{logsob1} implies  the Talagrand inequality:
\begin{align}\label{Talag}\W_2(\nu,\bar{\mu})^2\leq 4 C_1 \mathrm{Ent}(\nu|\bar{\mu}),\ \ \nu\in\scr P_2(\R^{2d}).
\end{align}
 For simplicity, denote $\mu_t=\L_{(X_{t}^{\mu_0},Y_t^{\mu_0})}$. Finally, for any $t\geq 1$, by \eqref{Talag}, \eqref{entc1} and \eqref{enc}, we have
\begin{align*}
\W_2(\L_{(X_t^{\mu_0},Y_t^{\mu_0})},\bar{\mu})^2&\leq 4C_1 \mathrm{Ent}(\L_{(X_t^{\mu_0},Y_t^{\mu_0})}|\bar{\mu})\\
%=c_2\mathrm{Ent}(\L_{(X_{t-1} ^{\mu_{1}},Y_{t-1}^{\mu_{1}})}|\bar{\mu})\\
&\leq 4C_1 c\e^{-2\lambda (t-1)}\mathrm{Ent}(\L_{(X_{1}^{\mu_0},Y_1^{\mu_0})}|\bar{\mu})\\
&\leq 4C_1c\e^{-2\lambda (t-1)}c_1\W_2(\mu_0,\bar{\mu})^2.
\end{align*}
Combining this with \eqref{w2sto} completes the proof of \eqref{w2con1}, hence of Theorem~\ref{hyc}.
\end{proof}

\subsection{Proof of Theorem \ref{cty}}
%Let $b:\R^{2d}\times \scr P_2(\R^{2d})\to\R^d$, $\sigma:\scr P_2(\R^{2d})\to\R^d\otimes\R^n$ be measurable and bounded on bounded sets. Consider
%\begin{align}\label{Degth}\left\{
  %\begin{array}{ll}
    %\d X_t=Y_t\d t,  \\
    %\d Y_t=b(X_t,Y_t,\L_{(X_t,Y_t)})\d t+\sigma(\L_{(X_t,Y_t)})\d W_t.
  %\end{array}
%\right.
%\end{align}

%\begin{align}\label{Degde}\left\{
%  \begin{array}{ll}
%    \d X_t^\mu=Y_t^\mu\d t,  \\
%    \d Y_t^\mu=-Y_t^\mu+b(X_t^\mu,P_t^\ast\mu_0)\d t+\sigma(P_t^\ast\mu_0)\d W_t,
%  \end{array}
%\right.
%\end{align}

In this section, we will prove exponential ergodicity in relative entropy and $\W_2$ for the McKean-Vlasov SDE \eqref{Degth}.
Recall that $P_t^\ast \mu_0$ denotes the distribution of the solution to \eqref{Degth} with initial distribution $\mu_0\in\scr P_2(\R^{2d})$.

For any $\mu\in\scr P_2(\R^{2d})$, consider the time-homogeneous decoupled SDEs
\begin{align}\label{tih12}\left\{
  \begin{array}{ll}
    \d \tilde{X}_t^\mu=\tilde{Y}_t^\mu\d t,  \\
    \d \tilde{Y}_t^\mu=b(\tilde{X}_t^\mu,\tilde{Y}_t^\mu,\mu)\d t+\sigma(\mu)\d W_t.
  \end{array}
\right.
\end{align}
Under {\bf(C)}, \eqref{tih12} is well-posed and has a unique invariant probability measure $\Phi(\mu)\in\scr P_2(\R^{2d})$, see for instance \cite[Theorem 3.3]{HKR25}. Let $\tilde{P}_t^\mu$ be the associated semigroup to \eqref{tih12} and $(\tilde{P}_t^\mu)^\ast\mu_0$ be the distribution of the solution $\tilde{X}_t^\mu$ to \eqref{tih12} with initial distribution $\mu_0\in\scr P(\R^{2d})$. For any $M>0$, let $$\hat{\scr P}_{2,M}(\R^{2d}):=\left\{\mu\in\scr P_2(\R^{2d}): \|\mu\|_2^2\leq  M \right\}.$$

We first give a log-Sobolev inequality for $\Phi(\mu)$ uniform over $\mu\in \hat{\scr P}_{2,M}(\R^{2d})$, which is a corollary of \cite[Theorem 3.5(2)]{HKR25}.
\begin{lem}\label{lemLSIPhimu}
 Assume {\bf (C)}. Then for any $M>0$,
 the log-Sobolev inequality holds:
\begin{align}
\label{logso}\nonumber&\Phi(\mu)(f\log f)\leq C_{\mathrm{LS}}(M)\Phi(\mu)(|\nabla \sqrt{f}|^2),\\
&\qquad\quad f\in C_0^\infty(\R^{2d}), f>0,\Phi(\mu)(f)=1, \mu\in\hat{\scr P}_{2,M}(\R^{2d}),
\end{align}
where $C_{\mathrm{LS}}(M)$ is a constant depending on $M$, $d$ and the constants in {\bf (C)}.
\end{lem}

 \begin{proof} By \cite[Theorem 3.5(2)]{HKR25}, it is sufficient to prove that for any $R>0$ and $M>0$, there exist positive constants $C_1$ and $C_2$ depending on $R$, $M$ such that
\begin{align}\label{upbloh}C_1\leq \rho_\mu(x)\leq C_2,\ \ x\in B_R, \mu\in\hat{\scr P}_{2,M}(\R^{2d})
\end{align}
for $\rho_\mu:=\frac{\d \Phi(\mu)}{\d x}$.
For $\lambda >0$, let
$$g_{\lambda}(1,z)=\e^{-\frac{|z|^2}{2\lambda}},\ \ z\in\R^{2d}.$$
For any $\mu\in\hat{\scr P}_{2,M}(\R^{2d})$, let $\theta_t^\mu(z)=(\theta_t^{(1),\mu}(z),\theta_t^{(2),\mu}(z))$ solve
\begin{align*}\left\{
  \begin{array}{ll}
    \d \theta_t^{(1),\mu}(z)=\theta_t^{(2),\mu}(z)\d t,  \\
    \d \theta_t^{(2),\mu}(z)=b(\theta_t^\mu(z),\mu)\d t,\ \ \theta_0^\mu(z)=z\in\R^{2d}.
  \end{array}
\right.
\end{align*}
Since \begin{align}\label{bbf}
|b(z,\mu)|\leq K_b|z|+K_I \sqrt{M}+|b(0,\delta_0)|,\ \ \mu\in\hat{\scr P}_{2,M}(\R^{2d}), z\in\R^{2d},
\end{align}
we conclude that
\begin{align}\label{theta}
|\theta_1^\mu(z)|\leq \bar{C}|z|+C_M,\ \ \mu\in\hat{\scr P}_{2,M}(\R^{2d}), z\in\R^{2d}
\end{align}
for some constant $C_M>0$ depending on $M$.
Moreover, by \eqref{bbf}, {\bf(C)} and \cite[Theorem 1.1(1)]{CMPZ}, there exist $C_0>1$ and $\lambda_0>0$ depending on $M$ such that the heat kernel $\tilde{p}_t^\mu(z,\tilde{z})$ associated to \eqref{tih12} satisfies
\begin{align}\label{lup}&\nonumber C_0^{-1}g_{\lambda_0^{-1}}(1,\tilde{z}-\theta_1^\mu(z))\leq \tilde{p}_1^\mu(z,\tilde{z})\leq C_0g_{\lambda_0}(1,\tilde{z}-\theta_1^\mu(z)),\\
&\ \ z,\tilde{z}\in\R^{2d}, \mu\in\hat{\scr P}_{2,M}(\R^{2d}).
\end{align}
Moreover, it follows from \cite[(3.12)]{HKR25} that $(\Phi(\mu))_{\mu\in\hat{\scr P}_{2,M}(\R^{2d})}$ is tight so that there exists a constant $R_0>0$ such that
$$\int_{B_{R_0}}\rho_\mu(z)\d z>\frac{1}{2}, \ \ \mu\in\hat{\scr P}_{2,M}(\R^{2d}).$$
Note that
$$\rho_\mu(\tilde{z})=\int_{\R^{2d}}\tilde{p}_1^\mu(z,\tilde{z})\rho_\mu(z)\d z,$$
and for any $ z,\tilde{z}\in\R^{2d},\mu\in\hat{\scr P}_{2,M}(\R^{2d})$,
$$0\leq|\tilde{z}-\theta_1^\mu(z)|^2\leq 2|\tilde{z}|+2|\theta_1^\mu(z)|^2\leq \tilde{C}_N+C_3|\tilde{z}|^2+C_3|z|^2.$$
We derive from \eqref{lup} that $\rho_\mu(\tilde{z})\leq C_0$ and
\begin{align}\label{cmt}
\rho_\mu(\tilde{z})\geq \int_{B_{R_0}}\tilde{p}_1^\mu(z,\tilde{z})\rho_\mu(z)\d z\geq C_1(M,R),\ \ \tilde{z}\in B_R,\mu\in\hat{\scr P}_{2,M}(\R^{2d}).
\end{align}
Hence, we obtain \eqref{upbloh} and the proof is complete.
\end{proof}

 With the above lemma in hand, we now present a result on the uniform $L^2$-exponential ergodicity in $\mu\in\hat{\scr P}_{2,M}(\R^{2d})$ for any $M>0$.
\begin{lem}\label{conL2} Assume {\bf (C)}.
Then for any $M>0$, there exist constants $c_0>0,\lambda_0>0$ depending on $M$, $d$ and the constants in {\bf (C)} such that
\begin{align*}\|\tilde{P}_t^\mu f-\Phi(\mu)(f)\|_{L^2(\Phi(\mu))}\leq c_0\e^{-\lambda_0 t}\|f-\Phi(\mu)(f)\|_{L^2(\Phi(\mu))},\ \ t\geq 0, \mu\in\hat{\scr P}_{2,M}(\R^{2d}).
\end{align*}
\end{lem}
\begin{proof} Firstly, it holds
\begin{align}\label{lipb1}|b(z,\mu)-b(z,\bar{\mu})|\leq K_b|z-\bar{z}|, \ \ \mu\in \scr P_2(\R^{2d}).
\end{align}
Second, \eqref{logso} implies the Poincar\'{e} inequality
$$\Phi(\mu)(f^2)-\Phi(\mu)(f)^2\leq C_{\mathrm{PI}}(M)\Phi(\mu)(|\nabla f|^2),\ \ f\in C_0^\infty(\R^{2d}), \mu\in\hat{\scr P}_{2,M}(\R^{2d}).$$
This combined with \eqref{lipb1}, \eqref{eup} and Lemma \ref{ExpL2} completes the proof.
\end{proof}

\begin{proof}[Proof of Theorem \ref{cty}] ({\bf Step 1: Existence and Uniqueness of invariant probability measure})

(i) Thanks to \cite[Theorem 3.3]{HKR25}, there exists a constant $\eta_0>0$ such that when $K_I<\eta_0$, we can find a constant $M>0$ such  that $\hat{\scr P}_{2,M}(\R^{2d})$  is fixed by the map $\Phi$. By \cite[Theorem 3.4]{HKR25}, there exist constants $t_0>0$ and $C_0>0$ such that
\begin{align}\label{hyper413}
\|\tilde{P}_{t_0}^\mu\|_{L^2(\Phi(\mu))\to L^{4}(\Phi(\mu))}\leq C_0, \ \ \mu\in\hat{\scr P}_{2,M}(\R^{2d}).
\end{align}
Moreover, $t_0$ and $C_0$ can be chosen independently from $K_I \in [0,\eta_0)$, since if the condition {\bf (C)} is satisfied for some $K_I < \eta_0$ then it is also satisfied with $K_I$ replaced by $\eta_0$.
This together with Lemmas \ref{conL2} and \ref{Hypercon} implies that for large enough $t_1\geq t_0$ (again, independent from $K_I$),
\begin{align}\label{hyper4p}
\|\tilde{P}_{t_1}^\mu\|_{L^2(\Phi(\mu))\to L^{4}(\Phi(\mu))}= 1, \ \ \mu\in\hat{\scr P}_{2,M}(\R^{2d}).
\end{align}
Let $(\tilde{P}_t^\mu)^\dagger$ stand for the adjoint operator of $\tilde{P}_t^\mu$ in $L^2(\Phi(\mu))$.
By the same argument to derive \eqref{hyper5} from \eqref{hyperne}, we obtain
\begin{align}\label{hyper41}
\|(\tilde{P}_{t_1}^\mu)^\dagger\|_{L^{\frac{4}{3}}(\Phi(\mu))\to L^{2}(\Phi(\mu))}= 1, \ \ \mu\in\hat{\scr P}_{2,M}(\R^{2d}).
\end{align}
So, by Lemma \ref{Hypercon}, we can find constants $c,\lambda>0$ (independent from $K_I$), such that, for $K_I\in[0,\eta_0)$,
\begin{align}\label{pty1}\mathrm{Ent}((\tilde{P}_t^\mu)^\ast\mu_0|\Phi(\mu))\leq c\e^{-2\lambda t}\mathrm{Ent}(\mu_0|\Phi(\mu)), \ \ t\geq 0, \mu_0\in\scr P_2(\R^{2d}),\mu\in\hat{\scr P}_{2,M}(\R^{2d}).
\end{align}
It now only remains to repeat the proof of \cite[Theorem 3.9]{HKR25}. To be self-contained, we list the main procedure here.

(ii) Reasonning as when we derived \eqref{w2con1} from \eqref{entc1}, we deduce from \eqref{pty1} that
\begin{align}\label{pty}\W_2((\tilde{P}_t^\mu)^\ast\mu_0,\Phi(\mu))\leq \tilde{c}\e^{-\lambda t}\W_2(\mu_0,\Phi(\mu)), \ \ t\geq 0, \mu_0\in\scr P_2(\R^{2d}),\mu\in\hat{\scr P}_{2,M}(\R^{2d}),
\end{align}
with $\tilde c$ independent from $K_I\in[0,\eta_0)$.
By \eqref{Lipsc}, it is not difficult to see from Gronwall's inequality that
\begin{align}\label{samin}\W_2((\tilde{P}_t^\mu)^\ast\Phi(\nu),(\tilde{P}_t^\nu)^\ast\Phi(\nu))^2\leq \e^{(2+2K_b)t}\int_0^tK_I^2\W_2(\mu,\nu)^2\d s,\ \ \mu,\nu\in\scr P_2(\R^{2d}).
\end{align}
Combining this with \eqref{pty} for $\mu_0=\Phi(\nu)$, we have
\begin{align*}&\W_2(\Phi(\mu),\Phi(\nu)) \leq \tilde{c}\e^{-\lambda t}\W_2(\Phi(\mu),\Phi(\nu))+\e^{(1+K_b)t}\sqrt{t}K_I\W_2(\mu,\nu),\ \ \mu,\nu\in\hat{\scr P}_{2,M}(\R^d).
\end{align*}
Let
$$\eta_1=\left(\inf_{t>\frac{\log {\tilde{c}}}{\lambda}} \frac{\e^{(1+K_b)t}\sqrt{t}}{1-\tilde{c}\e^{-\lambda t}}\right)^{-1}.$$
When $K_I<\min(\eta_0,\eta_1)$, the Banach fixed theorem implies that \eqref{Degth} has a unique invariant probability measure $\bar{\mu}\in\hat{\scr P}_{2,M}(\R^{2d})$.

({\bf Step 2: Exponential ergodicity in $L^2$-Wasserstein distance and relative entropy})

Again by \eqref{Lipsc} and Gronwall's inequality, we get
$$\W_2(P_t^\ast\mu_0,(\tilde{P}_t^{\bar{\mu}})^\ast\mu_0)^2\leq \e^{(2+2K_b)t}\int_0^tK_I^2\W_2(P_s^\ast\mu_0,\bar{\mu})^2\d s,\ \ \mu_0\in\scr P_2(\R^{2d}).
$$
This together with \eqref{pty} for $\mu=\bar{\mu}$, the triangle inequality and Gronwall's inequality yields
\begin{align}\label{tri}
\W_2(P_t^\ast\mu_0,\bar{\mu})^2\leq \left(K_I^2\e^{C_0 t}+2\tilde{c}^2\e^{-2\lambda t}\right)\W_2(\mu_0,\bar{\mu})^2,\ \ \mu_0\in\scr P_2(\R^{2d})
\end{align}
for some constant $C_0>0$ independent from $K_I$. This means that we can find $\hat{t}>0$ and $\alpha\in(0,1)$ (independent from $K_I$) such that, when $K_I$ is small enough,
\begin{align*}\W_2(P_{\hat{t}}^\ast\mu_0,\bar{\mu})\leq \alpha\W_2(\mu_0,\bar{\mu}), \ \ \mu_0\in\scr P_2(\R^{2d}).
\end{align*}
By the semigroup property $P_t^\ast P_s^\ast=P_{s+t}^\ast$ and \eqref{tri}, we get \eqref{w2con}. Next, by \eqref{logso} for $\mu=\bar{\mu}$, the Talagrand inequality holds, i.e.
\begin{align}\label{w2h}\W_2(\mu_0,\bar{\mu})^2\leq \hat{c}\mathrm{Ent}(\mu_0|\bar{\mu}),\ \ \mu_0\in\scr P_2(\R^{2d}),
\end{align}
for some $\hat c>0$ independent from $K_I$. Combining the log-Harnack inequality in \cite{HW24} with \eqref{w2con} and \eqref{w2h}, one may deduce, for some constants $c_3,c_4,c_5>0$ independent from $K_I$,
\begin{align*}
\mathrm{Ent}(P_t^\ast\mu_0|\bar{\mu})=\mathrm{Ent}(P_1^\ast P_{t-1}^\ast\mu_0|P_1^\ast\bar{\mu})&\leq c_3\W_2(P_{t-1}^\ast\mu_0,\bar{\mu})^2\\
&\leq c_4c\e^{-2\lambda (t-1)}\W_2(\mu_0,\bar{\mu})^2\\
&\leq c_5\e^{-2\lambda (t-1)}\mathrm{Ent}(\mu_0|\bar{\mu}),\ \ t\geq 1, \mu_0\in\scr P_2(\R^{2d}).
\end{align*}
Hence, the proof is completed.
\end{proof}

\subsection{Proof of Theorem \ref{cty12pn}}
%In this section, we consider the mean field interacting particle, i.e.
%\begin{align}\label{Degthp}\left\{
  %\begin{array}{ll}
   % \d X_t^{i,N}=Y_t^{i,N}\d t,  \\
    %\d Y_t^{i,N}=b\left(X_t^{i,N},Y_t^{i,N},\frac{1}{N}\sum_{i=1}^N\delta_{(X_t^{i,N},Y_t^{i,N})}\right)\d t+\sigma\d W_t^i,
  %\end{array}
%\right.
%\end{align}
%where $(W_t^i)_{i\geq 1}$ are independent $n$-dimensional Brownian motions.
%In this case, $\sigma$ cannot depend on the measurable variable since the Harnack inequality for multiplicative kinetic SDEs is still open.
%The corresponding McKean-Vlasov SDE is
%\begin{align}\label{Degth3}\left\{
  %\begin{array}{ll}
    %\d X_t=Y_t\d t,  \\
    %\d Y_t=b(X_t,Y_t,\L_{(X_t,Y_t)})\d t+\sigma\d W_t.
  %\end{array}
%\right.
%\end{align}

In this section we prove Theorem \ref{cty12pn}.
First, we reformulate \eqref{Degthp}. To this end, define
$$b^N(\mathbf{z})=\left(b\left(z^i,\frac{1}{N}\sum_{i=1}^N\delta_{z^i}\right)\right)_{1\leq i\leq N},\ \ \mathbf{z}=(z^1,z^2,\cdots, z^N)\in(\R^{2d})^N.$$
Let $\sigma^N=\mathrm{diag}(\sigma_i)_{1\leq i\leq N}$ with $\sigma_i=\sigma$ and $\mathbf{W}_t=(W_t^i)_{1\leq i\leq N}$.
Then \eqref{Degthp} can be reformulated as
\begin{align}\label{DegthpN}\left\{
  \begin{array}{ll}
    \d \mathbf{X}_t=\mathbf{Y}_t\d t,  \\
    \d \mathbf{Y}_t=b^N\left(\mathbf{X}_t,\mathbf{Y}_t\right)\d t+\sigma^N\d \mathbf{W}_t.
  \end{array}
\right.
\end{align}
It is easy to see from \eqref{Lipsc} that
\begin{align}\label{Lipsc3}
\nonumber&|b^N(\mathbf{z})-b^N(\bar{\mathbf{z}})|^2\\
\nonumber&= \sum_{i=1}^N\left|b\left(z^i,\frac{1}{N}\sum_{i=1}^N\delta_{z^i}\right)- b\left(\bar{z}^i,\frac{1}{N}\sum_{i=1}^N\delta_{\bar{z}^i}\right)\right|^2\\
&\leq \sum_{i=1}^N \left(2K_b^2|z_i-\bar{z}_i|^2+2K_I^2\W_2\left(\frac{1}{N}\sum_{i=1}^N\delta_{z^i}, \frac{1}{N}\sum_{i=1}^N\delta_{\bar{z}^i}\right)^2\right)\\
\nonumber&\leq \sum_{i=1}^N \left(2K_b^2|z_i-\bar{z}_i|^2+2K_I^2\frac{1}{N}\sum_{i=1}^N|z^i-\bar{z}^i|^2\right)\\
\nonumber&=(2K_b^2+2 K_I^2)\sum_{i=1}^N|z^i-\bar{z}^i|^2,\ \ \mathbf{z}=(z^i)_{1\leq i\leq N}, \bar{\mathbf{z}}=(\bar{z}^i)_{1\leq i\leq N}\in(\R^{2d})^N.
\end{align}
This implies that under \eqref{Lipsc}, \eqref{Degthp} is well-posed and we denote by $P^N_t$ the associated semigroup. Let $(P_t^N)^\ast\nu^N$ be the distribution of the solution to \eqref{Degthp} with initial distribution $\nu^N\in\scr P((\R^{2d})^N)$.

Let $\mathbf{x}=(x^i)_{1\leq i\leq N}$, $\mathbf{y}=(y^i)_{1\leq i\leq N}$, $\bar{\mathbf{x}}=(\bar{x}^i)_{1\leq i\leq N}$, $\bar{\mathbf{y}}=(\bar{y}^i)_{1\leq i\leq N}\in(\R^{d})^N$.
$\mathbf{z}=(\mathbf{x},\mathbf{y})$, $\bar{\mathbf{z}}=(\bar{\mathbf{x}},\bar{\mathbf{y}})$.
Firstly, it follows from \eqref{Patdi} and \eqref{Lipsc} that
\begin{align}\label{PatdiN}
\nonumber&\<r^2(\mathbf{x}-\bar{\mathbf{x}})+rr_0(\mathbf{y}-\bar{\mathbf{y}}),\mathbf{y}-\bar{\mathbf{y}}\>\\
\nonumber&+\<(\mathbf{y}-\bar{\mathbf{y}})+rr_0(\mathbf{x}-\bar{\mathbf{x}}),b^N(\mathbf{x},\mathbf{y})-b^N(\bar{\mathbf{x}}, \bar{\mathbf{y}})\>\\
\nonumber&=\sum_{i=1}^N\left\<r^2(x^i-\bar{x}^i)+rr_0(y^i-\bar{y}^i),y^i-\bar{y}^i\right\>\\
&+\sum_{i=1}^N\<(y^i-\bar{y}^i)+rr_0(x^i-\bar{x}^i),b(x^i,y^i,\frac{1}{N}\sum_{i=1}^N\delta_{(x^i,y^i)})- b(\bar{x}^i,\bar{y}^i,\frac{1}{N}\sum_{i=1}^N\delta_{(x^i,y^i)})\>\\
\nonumber&+\sum_{i=1}^N\<(y^i-\bar{y}^i)+rr_0(x^i-\bar{x}^i),b(\bar{x}^i,\bar{y}^i,\frac{1}{N}\sum_{i=1}^N\delta_{(x^i,y^i)})- b(\bar{x}^i,\bar{y}^i,\frac{1}{N}\sum_{i=1}^N\delta_{(\bar{x}^i,\bar{y}^i)})\>\\
\nonumber&\leq\sum_{i=1}^N (-\theta+K_I(1+|rr_0|))(|x^i-\bar{x}^i|^2+|y^i-\bar{y}^i|^2)\\
\nonumber&+\sum_{i=1}^N(C_0+\theta)(|x^i-\bar{x}^i|^2+|y^i-\bar{y}^i|^2)1_{\{|x^i-\bar{x}^i|^2+|y^i-\bar{y}^i|^2< R^2\}}\\
\nonumber&\leq (-\theta+K_I(1+|rr_0|))|\mathbf{z}-\bar{\mathbf{z}}|^2+(C_0+\theta)|\mathbf{z}-\bar{\mathbf{z}}|^21_{\{|\mathbf{z}-\bar{\mathbf{z}}|^2< NR^2\}}.
\end{align}
So, it follows from \cite[Proof of Theorem 3.3]{HKR25} that there exists a constant $\mathbf{K}_1>0$ independent of $N$ such that when $K_I<\mathbf{K}_1$, \eqref{Degthp} has a unique invariant probability measure $\mu^N\in\scr P_2((\R^{2d})^N)$.

\subsubsection{Uniform in $N$ hypercontractivity for referenced $N$-independent particles}
Consider
\begin{align}\label{refmv}\left\{
  \begin{array}{ll}
    \d \hat{X}_t=\hat{Y}_t\d t,  \\
    \d \hat{Y}_t=\hat{b}_N(\hat{X}_t,\hat{Y}_t, \L_{(\hat{X}_t,\hat{Y}_t)})\d t+\sigma\d W_t,
  \end{array}
\right.
\end{align}
where $$\hat{b}_N(z,\mu)=\int_{(\R^{d})^{N-1}}b(z,\frac{1}{N}\delta_z+\frac{1}{N}\sum_{j=1}^{N-1}\delta_{z_j})\mu^{ \otimes (N-1)}(\d z_1,\d z_2,\cdots,\d z_{N-1}).$$
Then it is not difficult to see that $\hat{b}_N(0,\delta_0)=b(0,\delta_0)$, and $\hat{b}_N$ in place of $b$ satisfies the condition {\bf(C)} with $(K_b,\theta)$ replacing by $(K_b+\frac{K_I}{\sqrt{N}}, \theta-\frac{K_I}{\sqrt{N}}\sqrt{1+(rr_0)^2})$. According to Theorem \ref{cty}, there exists a constant $\mathbf{K}_2\leq \mathbf{K}_1$ such that when $K_I<\mathbf{K}_2$, \eqref{refmv} has a unique invariant probability measure denoted by $\nu_N$.
Consider the referenced SDE
\begin{align}\label{tih13eq1}\left\{
  \begin{array}{ll}
    \d \hat{X}_t^{\nu_N,i}=\hat{Y}_t^{\nu_N,i}\d t,  \\
    \d \hat{Y}_t^{\nu_N,i}=\hat{b}_N(\hat{X}_t^{\nu_N,i},\hat{Y}_t^{\nu_N,i},\nu_N)\d t+\sigma\d W_t^i,\ \ 1\leq i\leq N.
  \end{array}
\right.
\end{align}
Let $\{\hat{P}_t^{N,1}\}_{t\geq 0}$ be the semigroup associated to $(\hat{X}_t^{\nu_N,i},\hat{Y}_t^{\nu_N,i})$ and  $(\hat{P}_{t_1}^{N,1})^\dagger$ be the adjoint operator of $\hat{P}_t^{N,1}$ in $L^2(\nu_N)$. Let $\{\hat{P}_t^N\}_{t\geq 0}$ be the $N$-independent semigroup of $\{\hat{P}_t^{N,1}\}_{t\geq 0}$, i.e. the semigroup of \eqref{tih13eq1} and  $(\hat{P}_{t}^{N})^\dagger$ be the adjoint operator of $\hat{P}_t^{N}$ in $L^2(\nu_N^{\otimes N})$, where $\nu_N^{\otimes N}$ is the $N$-independent product of $\nu_N$.

%and for any $\nu^N\in\scr P((\R^{2d})^N)$, define
%$$(\nu^N(\tilde{P}_t^{\bar{\mu}})^{\otimes N})(F)=\int_{(\R^{2d})^N}(\tilde{P}_t^{\bar{\mu}})^{\otimes N}(F)\d \nu^N,\ \ F\in\scr B_b((\R^{2d})^N).$$

%\begin{align*}
%\W_2((P_t^N)^\ast\nu^N,\bar{\mu}^N)&\leq \W_2((P_t^N)^\ast\nu^N,\nu^N(\tilde{P}_t^{\bar{\mu}})^{\otimes N})+\W_2(\nu^N(\tilde{P}_t^{\bar{\mu}})^{\otimes N}, \bar{\mu}^{\otimes N})+\W_2(\bar{\mu}^N,\bar{\mu}^{\otimes N})\\
%&=:I_1+I_2+I_3.
%\end{align*}
By the same argument to derive \eqref{hyper41}, there exists a constant $\mathbf{K}_3\leq \mathbf{K}_2$ depending on $d$ and the constants in~{\bf (C)} except $K_I$, such that when $K_I<\mathbf{K}_3$, we may find a $t_1>0$ independent of $N$ such that
$$\|(\hat{P}_{t_1}^{N,1})^\dagger\|_{L^\frac{4}{3}(\nu_N)\to L^2(\nu_N)}=1.$$
We first claim the tensor property of hypercontractivity, i.e.
\begin{align}\label{tenco}\|(\hat{P}_{t_1}^{N})^\dagger\|_{L^\frac{4}{3}(\nu_N^{\otimes N})\to L^2(\nu_N^{\otimes N})}=\|[(\hat{P}_{t_1}^{N,1})^\dagger]^{\otimes N}\|_{L^\frac{4}{3}(\nu_N^{\otimes N})\to L^2(\nu_N^{\otimes N})}=1.
\end{align}
By the method of induction, it is sufficient to prove the case $N=2$.
Indeed, for any $g\in L^\frac{4}{3}(\nu_2^{\otimes 2})$, let  $g_y(z)=(\hat{P}_{t_1}^{2,1})^\dagger g(z,\cdot)(y)$. Then it holds
\begin{align*}&|[(\hat{P}_{t_1}^{2,1})^\dagger]^{\otimes 2}\|_{L^\frac{4}{3}(\nu_2^{\otimes 2})\to L^2(\nu_2^{\otimes 2})}^2\\
&=\sup_{\|g\|_{ L^\frac{4}{3}(\nu_2^{\otimes 2})}\leq 1}\int_{\R^{2d}}\int_{\R^{2d}}|[(\hat{P}_{t_1}^{2,1})^\dagger g_y](x)|^2\nu_2(\d x)\nu_2(\d y)\\
&\leq \sup_{\|g\|_{ L^\frac{4}{3}(\nu_2^{\otimes 2})}\leq 1}\int_{\R^{2d}}\left(\int_{\R^{2d}}|(\hat{P}_{t_1}^{2,1})^\dagger g(x,\cdot)(y)|^\frac{4}{3}\nu_2(\d x)\right)^{\frac{3}{2}}\nu_2(\d y)\\
&\leq \sup_{\|g\|_{ L^\frac{4}{3}(\nu_2^{\otimes 2})}\leq 1}\left(\int_{\R^{2d}}\left(\int_{\R^{2d}}|(\hat{P}_{t_1}^{2,1})^\dagger g(x,\cdot)(y)|^2\nu_2(\d y)\right)^{\frac{2}{3}}\nu_2(\d x)\right)^\frac{3}{2}\\
&\leq  \sup_{\|g\|_{ L^\frac{4}{3}(\nu_2^{\otimes 2})}\leq 1}\left(\int_{\R^{2d}}\int_{\R^{2d}}g(x,y)|^\frac{4}{3}\nu_2(\d y)\nu_2(\d x)\right)^\frac{3}{2}\\
&\leq 1,
\end{align*}
where in the second inequality, we used Minkowski's inequality.
\subsubsection{Uniform in $N$ hypercontractivity for mean field particle system}
In this part, we will use \eqref{tenco} as well as Girsanov's transform to derive the uniform in $N$ hypercontractivity for the mean field interacting particle system \eqref{Degthp}. We use the notations of the previous section, in particular $\nu_N$ is the stationary solution to~\eqref{refmv}.

Let $$G_{N,i}(z_1,z_2,\cdots,z_N)=b(z_i,\frac{1}{N}\sum_{j=1}^N\delta_{z_j})-\hat{b}_N(z_i,\nu_N),\ \ z_i\in\R^{2d}, 1\leq i\leq N. $$
Set $G_N=(G_{N,i})_{1\leq i\leq N}$.
Then we have
$$\nu_N^{\otimes N}(G_{N,i}|z_i)=0,\ \ z_i\in\R^{2d}, 1\leq i\leq N.$$
It follows from Lemma~\ref{lemLSIPhimu} that $\nu_N$ satisfies log-Sobolev inequality:
$$\nu_N(f^2\log f^2)-\nu_N(f^2)\log \nu_N(f^2)\leq 2c\nu_N(|\nabla f|^2), \ \ f\in C_0^\infty(\R^{2d}),$$
where $c$ is a constant independent of $N$.
For $1\leq i\leq N$ and $z_i\in\R^{2d}$, let
$$g_{z_i}((z_j)_{j\neq i})=b(z_i,\frac{1}{N}\sum_{j=1}^N\delta_{z_j}).$$
Noting that the Lipschitz constant of $g_{z_i}$ is $\frac{K_I}{\sqrt{N}}$, by the Poincar\'{e} inequality for $\nu_N^{\otimes N}(\cdot|z_i)$, we have
$$\nu_N^{\otimes N}(|G_{N,i}|^2|z_i)\leq \frac{cdK_I^2}{N}, \ \ 1\leq i\leq N.$$
%So,
%$$\nu_N^{\otimes N}(|G_{N}|^2)\leq cdK_I^2.$$
For any $a>0$ satisfying $4acK_I^2<1$, the concentration for $\nu_N^{\otimes N}(\cdot|z_i)$  implied by the log-Sobolev inequality gives
\begin{align}\label{conex}\nu_N^{\otimes N}(\e^{aN|G_{N,i}|^2}|z_i)\leq \frac{\e^{2acdK_I^2}}{\sqrt{1-4acK_I^2}},
\end{align}
see e.g. \cite[Proposition 5.4.1]{BGL}.
It follows from Jensen's inequality that
$$\e^{a|G_{N}|^2}=\e^{a\sum_{i=1}^N|G_{N,i}|^2}\leq \frac{1}{N}\sum_{i=1}^N\e^{aN|G_{N,i}|^2},$$
which together with \eqref{conex} implies
\begin{align}\label{expro}\nu_N^{\otimes N}(\e^{a|G_{N}|^2})\leq \frac{\e^{2acdK_I^2}}{\sqrt{1-4acK_I^2}}.
\end{align}
Let $u_{N,i}=\sigma^\ast(\sigma\sigma^\ast)^{-1}G_{N,i}$ and $u_N=(u_{N,i})_{1\leq i\leq N}$. Combining \eqref{expro}, for $\tilde{a}$ satisfying $4\tilde{a}\delta_2^{-1}cK_I^2<1$, we derive
\begin{align}\label{ein}
\nu_N^{\otimes N}(\e^{\tilde{a}|u_{N}|^2})\leq \nu_N^{\otimes N}(\e^{\tilde{a}\delta_2^{-1}|G_{N}|^2})\leq \frac{\e^{2\tilde{a}\delta_2^{-1}cdK_I^2}}{\sqrt{1-4\tilde{a}\delta_2^{-1}cK_I^2}},\ \ N\geq 1.
\end{align}
By \eqref{tenco}, Riesz-Thorin's interpolation as well as the semigroup property, we may find $T>t_1$ independent of $N$ such that
\begin{align}\label{tenco3}\|\hat{P}_{T}^{N}\|_{L^\frac{10}{9}(\nu_N^{\otimes N})\to L^6(\nu_N^{\otimes N})}=1.
\end{align}
Hence, it holds
\begin{align}\label{Hy1}\|(\hat{P}_T^N)^\dagger\|_{L^2(\nu_N^{\otimes N})\to L^8(\nu_N^{\otimes N})}=1
\end{align}
and
\begin{align}\label{Hy2}\|\hat{P}_T^N\|_{L^{8/7}(\nu_N^{\otimes N})\to L^{40/7}(\nu_N^{\otimes N})}=1
\end{align}

Next, let $Z_t^\mathbf{z}$ solve \eqref{tih13eq1} with $Z_0^\mathbf{z}=\mathbf{z}\in(\R^{2d})^N$,
 and let
$$R_t^{\mathbf{z}}=\exp\left\{\int_0^t\sum_{i=1}^N\<u_{N,i}(Z_s^\mathbf{z}),\d W^i_s\>-\frac{1}{2}\int_0^t\sum_{i=1}^N|u_{N,i}(Z_s^\mathbf{z})|^2\d s\right\},\ \ t\geq 0.$$
It follows from Jensen's inequality that for any $\alpha>0$,
\begin{align*}&\int_{(\R^{2d})^N}\E\exp\left\{\alpha\int_0^T\sum_{i=1}^N|u_{N,i}(Z_s^\mathbf{z})|^2\d s\right\}\nu_N^{\otimes N}(\d \mathbf{z})\\
&\leq \int_{(\R^{2d})^N}\frac{1}{T}\E\int_0^T\exp\left\{\alpha T\sum_{i=1}^N|u_{N,i}(Z_s^\mathbf{z})|^2\right\}\d s\nu_N^{\otimes N}(\d \mathbf{z})\\
&=\nu_N^{\otimes N}(\e^{\alpha T|u_{N}|^2}).
\end{align*}
By \eqref{ein}, we may find a constant $\mathbf{K}_4\leq \mathbf{K}_3$ such that when $K_I<\mathbf{K}_4$, $\nu_N^{\otimes N}$-a.s. $\mathbf{z}$, $\{R_t^\mathbf{z}\}_{t\in[0,T]}$ is a martingale so that Girsanov's theorem implies $\nu_N^{\otimes N}$-a.s. $\mathbf{z}$,
\begin{align}\label{girtr}P_T^N f(\mathbf{z})=\E [R_T^\mathbf{z} f(Z_T^\mathbf{z})],
\end{align}
\begin{align}\label{RT1} \sup_{N\geq 1}\int_{(\R^{2d})^N}\E [(R_T^{\mathbf{z}})^{6}]\nu_N^{\otimes N}(\d \mathbf{z})\leq \sqrt{M},
\end{align}
\begin{align}\label{RT2}
\sup_{N\geq 1}\int_{(\R^{2d})^N}\E [(R_T^{\mathbf{z}})^{16}]\nu_N^{\otimes N}(\d \mathbf{z})\leq \sqrt{M},\ \
\end{align}
and
\begin{align}\label{Rlimt}\sup_{N\geq 1}\int_{(\R^{2d})^N}\E [(R_T^{\mathbf{z}})^{-24}]\nu_N^{\otimes N}(\d \mathbf{z})<\sqrt{M}
\end{align}
for $M\geq 1$ going to $1$ as $K_I$ goes to $0$. Let $h_N=\frac{\d \mu^N}{\d \nu_N^{\otimes N}}$.

{\bf (i)} We first prove
\begin{align}\label{polim}\nu_N^{\otimes N}(h_N^4)\leq M.\end{align}
It follows from the H\"{o}lder inequality and \eqref{girtr} that $\nu_N^{\otimes N}$-a.s. $\mathbf{z}$,
$$|P_T^N f(\mathbf{z})|\leq [\E (R_T^\mathbf{z})^6]^{\frac{1}{6}}[\hat{P}_{T}^{N}(|f|^{\frac{6}{5}})]^{\frac{5}{6}}(\mathbf{z}),$$
which together with H\"{o}lder's inequality, \eqref{RT1} and \eqref{tenco3} gives
\begin{align*}\|P_T^N f\|_{L^2(\nu_N^{\otimes N})}&\leq \left(\int_{(\R^{2d})^N}[\E (R_T^\mathbf{z})^6]^{\frac{1}{3}}[\hat{P}_{T}^{N}(|f|^{\frac{6}{5}})]^{\frac{5}{3}}(\mathbf{z})\nu_N^{\otimes N}(\d \mathbf{z})\right)^{\frac{1}{2}}\\
&\leq \left(\int_{(\R^{2d})^N}[\E (R_T^\mathbf{z})^6]\nu_N^{\otimes N}(\d \mathbf{z})\right)^{\frac{1}{6}}\left(\int_{(\R^{2d})^N}[\hat{P}_{T}^{N}(|f|^{\frac{6}{5}})]^{\frac{5}{2}}(\mathbf{z})\nu_N^{\otimes N}(\d \mathbf{z})\right)^{\frac{1}{3}}\\
&\leq M^{1/12}\||f|^{\frac{6}{5}}\|_{L^\frac{10}{9}(\nu_N^{\otimes N})}^{\frac{5}{6}}=M^{1/12}\|f\|_{L^\frac{4}{3}(\nu_N^{\otimes N})}.
\end{align*}
This implies
\begin{align*}\|(P_T^N)^\dagger f\|_{L^4(\nu_N^{\otimes N})}\leq M^{1/12}\|f\|_{L^2(\nu_N^{\otimes N})}.
\end{align*}
Combining this with the fact $$\|h_N\|_{L^2(\nu_N^{\otimes N})}=(\nu_N^{\otimes N}(h_N^{\frac{2}{3}}h_N^{\frac{4}{3}}))^{\frac{1}{2}}\leq \|h_N\|_{L^1(\nu_N^{\otimes N})}^{\frac{1}{3}}\|h_N\|_{L^4(\nu_N^{\otimes N})}^{\frac{2}{3}}=\|h_N\|_{L^4(\nu_N^{\otimes N})}^{\frac{2}{3}},$$ we conclude that
\begin{align*}\|(P_T^N)^\dagger f\|_{L^4(\nu_N^{\otimes N})}\leq M^{1/12}\|f\|_{L^4(\nu_N^{\otimes N})}^{\frac{2}{3}}.
\end{align*}
Taking $f=1$, and using iteration, we get
\begin{align}\label{sequence}\|[(P_T^N)^\dagger]^n 1\|_{L^4(\nu_N^{\otimes N})}\leq M^{1/4},\ \ n\geq 1.
\end{align}
%\textcolor{red}{By Ces\`{a}ro averages,
Consequently, $\{[(P_T^N)^\dagger]^n 1\}_{n\geq 1}$ is weakly compact in $L^4(\nu_N^{\otimes N})$, which implies that there exists a subsequence still denoted by $\{[(P_T^N)^\dagger]^n 1\}_{n\geq 1}$ such that as $n$ goes to infinity, $[(P_T^N)^\dagger]^n 1$ converges to $h_N$  weakly in $L^4(\nu_N^{\otimes N})$. By the lower semi-continuity of the norm $\|\cdot\|_{L^4(\nu_N^{\otimes N})}$ with respect to the weak topology in $L^4(\nu_N^{\otimes N})$, we get \eqref{polim} from \eqref{sequence} by taking $n\to\infty$.

{\bf (ii)} We intend to prove
\begin{align}\label{inlim}\lim_{K_I\to 0}\sup_{N}(\nu_N^{\otimes N}(h_N^{-4})-1)=0.
\end{align}

Denote $Z_t^{\nu_N^{\otimes N}}$ the solution to \eqref{tih13eq1} with initial distribution $\nu_N^{\otimes N}$, we first claim that
\begin{align}\label{impor}\E\left(R_T^{Z_0^{\nu_N^{\otimes N}}} h_N(Z_0^{\nu_N^{\otimes N}})\bigg|Z_T^{\nu_N^{\otimes N}}\right)=h_N(Z_T^{\nu_N^{\otimes N}}),
\end{align}
and
\begin{align}\label{inv}\E (f(Z_0^{\nu_N^{\otimes N}})|Z_T^{\nu_N^{\otimes N}})=(\hat{P}_T^N)^\dagger f(Z_T^{\nu_N^{\otimes N}}).
\end{align}
For \eqref{impor}, it is sufficient to prove that for any $\phi\in \scr B_b((\R^{2d})^N)$,
\begin{align}\label{ident}\E \left(R_T^{Z_0^{\nu_N^{\otimes N}}}  h_N(Z_0^{\nu_N^{\otimes N}})\phi(Z_T^{\nu_N^{\otimes N}})\right)=\E\left(h_N(Z_T^{\nu_N^{\otimes N}})\phi(Z_T^{\nu_N^{\otimes N}})\right).
\end{align}
The left hand side of \eqref{ident}is equal to
\begin{align*}&\E\left[h_N(\mathbf{z})\E (R_T^{\mathbf{z}}  \phi(Z_T^{\mathbf{z}}))\bigg|_{\mathbf{z}
=Z_0^{\nu_N^{\otimes N}}}\right]\\
&= \E[(h_NP_T^N\phi)(Z_0^{\nu_N^{\otimes N}})]=\nu_N^{\otimes N}(h_NP_T^N\phi)=\mu^N(P_T^N\phi)=\mu^N(\phi),
\end{align*}
while the right hand side of \eqref{ident} coincides with $\nu_N^{\otimes N}(h_N\phi)=\mu^N(\phi).$
So, \eqref{ident} holds. As to \eqref{inv}, it suffices to prove
\begin{align}\label{cmtyi}\E \left(f(Z_0^{\nu_N^{\otimes N}})\phi(Z_T^{\nu_N^{\otimes N}})\right)=\E [\phi(\hat{P}_T^N)^\dagger f](Z_T^{\nu_N^{\otimes N}}).
\end{align}
By taking conditional expectation, the left hand side of \eqref{cmtyi} is equal to $\E [(f\hat{P}_T^N\phi)(Z_0^{\nu_N^{\otimes N}})]$, which is not other but the right hand side by the definition of the adjoint operator.

Let $h_{N,\vv}=(h_N+\vv)^{-1}$ and $$G=\left(R_T^{Z_0^{\nu_N^{\otimes N}}} \right)^{-1}\vee 1.$$
Then it follows from \eqref{impor} and Jensen's inequality for conditional expectation that
\begin{align}\label{myu}h_{N,\vv}(Z_T^{\nu_N^{\otimes N}})\leq \E((R_T^{Z_0^{\nu_N^{\otimes N}}} h_N(Z_0^{\nu_N^{\otimes N}})+\vv)^{-1}|Z_T^{\nu_N^{\otimes N}})\leq \E\left(Gh_{N,\vv}(Z_0^{\nu_N^{\otimes N}})|Z_T^{\nu_N^{\otimes N}}\right).
\end{align}
By H\"{o}lder's inequality for conditional expectation and \eqref{inv}, one has
\begin{align*}&\E\left(Gh_{N,\vv}(Z_0^{\nu_N^{\otimes N}})|Z_T^{\nu_N^{\otimes N}}\right)\leq \left(\E(G^3|Z_T^{\nu_N^{\otimes N}})\right)^{1/3}\left(\E(h_{N,\vv}(Z_0^{\nu_N^{\otimes N}})^{\frac{3}{2}}|Z_T^{\nu_N^{\otimes N}})\right)^{2/3}\\
&=\left(\E(G^3|Z_T^{\nu_N^{\otimes N}})\right)^{1/3}\left((\hat{P}_T^{N})^\dagger h_{N,\vv}^{\frac{3}{2}}(Z_T^{\nu_N^{\otimes N}})\right)^{2/3}.
\end{align*}
So, we may use H\"{o}lder's inequality and \eqref{Hy1} to derive
\begin{align*}
&\left(\E\left(\E\left(Gh_{N,\vv}(Z_0^{\nu_N^{\otimes N}})|Z_T^{\nu_N^{\otimes N}}\right)\right)^4\right)^{1/4}\\
&\leq \left(\E\left(\E(G^3|Z_T^{\nu_N^{\otimes N}})\right)^{2}\right)^{1/6}\left(\E\left((\hat{P}_T^{N})^\dagger h_{N,\vv}^{\frac{3}{2}}(Z_T^{\nu_N^{\otimes N}})\right)^{8}\right)^{1/12}\\
&\leq (\E(G^6))^{1/6}\|(\hat{P}_T^{N})^\dagger h_{N,\vv}^{\frac{3}{2}}\|_{L^8(\nu_N^{\otimes N})}^{2/3}\leq (\E(G^6))^{1/6}\|h_{N,\vv}\|_{L^3(\nu_N^{\otimes N})}.
\end{align*}
This together with \eqref{myu} implies
\begin{align}\label{kpv}\|h_{N,\vv}\|_{L^4(\nu_N^{\otimes N})}\leq (\E(G^6))^{1/6}\|h_{N,\vv}\|_{L^3(\nu_N^{\otimes N})}.
\end{align}
For $0<a<1$, let $p_N(a)=\nu_N^{\otimes N}(h_N<a)$. By H\"{o}lder's inequality, we get
$$\|h_{N,\vv}1_{h_N<a}\|_{L^3(\nu_N^{\otimes N})}\leq p_N(a)^{1/12}\|h_{N,\vv}\|_{L^4(\nu_N^{\otimes N})},$$
which means
$$\|h_{N,\vv}\|_{L^3(\nu_N^{\otimes N})}\leq  a^{-1}+p_N(a)^{1/12}\|h_{N,\vv}\|_{L^4(\nu_N^{\otimes N})}.$$
This combined with \eqref{kpv} gives
$$\|h_{N,\vv}\|_{L^4(\nu_N^{\otimes N})}\leq  (\E(G^6))^{1/6}\left(a^{-1}+p_N(a)^{1/12}\|h_{N,\vv}\|_{L^4(\nu_N^{\otimes N})}\right).$$
Therefore, if $(\E(G^6))^{1/6}p_N(a)^{1/12}<1$, we derive from the monotone convergence theorem that
\begin{align}\label{inh}\|h_{N}^{-1}\|_{L^4(\nu_N^{\otimes N})}=\lim_{\vv\to0}\|h_{N,\vv}\|_{L^4(\nu_N^{\otimes N})}\leq \frac{(\E(G^6))^{1/6}a^{-1}}{1-(\E(G^6))^{1/6}p_N(a)^{1/12}}.
\end{align}
In view of $\nu_N^{\otimes N}(h_N)=1$ and \eqref{polim}, we arrive at
\begin{align}\label{pna}p_N(a)\leq \frac{\nu_N^{\otimes N}(h_N-1)^2}{(1-a)^2}=\frac{\nu_N^{\otimes N}(h_N^2)-1}{(1-a)^2}\leq \frac{\sqrt{M}-1}{(1-a)^2}.
\end{align}
Observing $\E\left(R_T^{Z_0^{\nu_N^{\otimes N}}}\right)^{-6}\geq 1$, we get from \eqref{Rlimt} that
\begin{align*}\E(G^6)&=\E\left(\left(R_T^{Z_0^{\nu_N^{\otimes N}}}\right)^{-6}\vee 1\right)
=1+\E\left(\left(R_T^{Z_0^{\nu_N^{\otimes N}}}\right)^{-6}-1\right)^{+}\\
&\leq 1+\sqrt{\E\left(\left(R_T^{Z_0^{\nu_N^{\otimes N}}}\right)^{-6}-1\right)^{2}}\leq 1+\sqrt{\E\left(R_T^{Z_0^{\nu_N^{\otimes N}}}\right)^{-12}-1}\leq 1+\sqrt{M^{1/4}-1}.
\end{align*}
Taking first $K_I\to0$(hence $M\to0$) and then $a\to1$ in \eqref{inh} and noting \eqref{pna}, we get \eqref{inlim}.

{\bf (iii)} We prove $\sup_{N}\|P_{2T}^N\|_{L^2(\mu^N)\to L^4(\mu^N)}=1$.

By the H\"{o}lder inequality and \eqref{Hy2}, it holds
\begin{align*}
&\|P_T^N\|_{L^2(\mu^N)\to L^4(\mu^N)}\\
&=\sup_{\|f\|_{L^2(\mu^N)}\leq 1}(\mu^N(|P_T^N f|^4))^{1/4}\\
&=\sup_{\|f\|_{L^2(\mu^N)}\leq 1}\left(\int_{(\R^{2d})^N}h_N(\mathbf{z})|\E R_T^\mathbf{z} f(Z_T^\mathbf{z})|^4\nu_N^{\otimes N}(\d \mathbf{z})\right)^{1/4}\\
&\leq \{\nu_N^{\otimes N}(h_N^4)\}^{1/16}\sup_{\|f\|_{L^2(\mu^N)}\leq 1}\left(\int_{(\R^{2d})^N}|\E R_T^\mathbf{z} f(Z_T^\mathbf{z})|^{16/3}\nu_N^{\otimes N}(\d \mathbf{z})\right)^{3/16}\\
&\leq \{\nu_N^{\otimes N}(h_N^4)\}^{1/16}\sup_{\|f\|_{L^2(\mu^N)}\leq 1}\left(\int_{(\R^{2d})^N}|[(\E| R_T^\mathbf{z}|^{7/2})^{2/7} (\E |f(Z_T^\mathbf{z})|^{7/5})^{5/7}]^{16/3}\nu_N^{\otimes N}(\d \mathbf{z})\right)^{3/16}\\
&\leq \{\nu_N^{\otimes N}(h_N^4)\}^{1/16}\left(\int_{(\R^{2d})^N}\E| R_T^\mathbf{z}|^{16}\nu_N^{\otimes N}(\d \mathbf{z}) \right)^{1/16}\left(\int_{(\R^{2d})^N}(\E |f(Z_T^\mathbf{z})|^{7/5})^{40/7}\nu_N^{\otimes N}(\d \mathbf{z})\right)^{1/8}\\
&\leq  \{\nu_N^{\otimes N}(h_N^4)\}^{1/16}\left(\int_{(\R^{2d})^N}\E| R_T^\mathbf{z}|^{16}\nu_N^{\otimes N}(\d \mathbf{z}) \right)^{1/16} \sup_{\|f\|_{L^2(\mu^N)}\leq 1} \left(\nu_N^{\otimes N}(|f|^{8/5}h_N^{4/5}h_N^{-4/5})\right)^{5/8}\\
&\leq \{\nu_N^{\otimes N}(h_N^4)\}^{1/16}\left(\int_{(\R^{2d})^N}\E| R_T^\mathbf{z}|^{16}\nu_N^{\otimes N}(\d \mathbf{z}) \right)^{1/16} \sup_{\|f\|_{L^2(\mu^N)}\leq 1} \left(\nu_N^{\otimes N}(|f|^{2}h_N)\right)^{1/2}(\nu_N^{\otimes N}(h_N^{-4}))^{1/8}\\
&=\{\nu_N^{\otimes N}(h_N^4)\}^{1/16}\left(\int_{(\R^{2d})^N}\E| R_T^\mathbf{z}|^{16}\nu_N^{\otimes N}(\d \mathbf{z}) \right)^{1/16}(\nu_N^{\otimes N}(h_N^{-4}))^{1/8}.
\end{align*}
This together with \eqref{Rlimt}, \eqref{polim} and \eqref{inlim} implies that $\sup_{N}\|P_T^N\|_{L^2(\mu^N)\to L^4(\mu^N)}\to 1$ as $K_I\to0$. So, we can choose $K_I$ small enough such that $\sup_{N}\|P_T^N\|_{L^2(\mu^N)\to L^4(\mu^N)}^4\leq 17/16$. This implies \begin{align}\label{gap}
\|P_T^N(f-\mu^N(f))\|_{L^2(\mu^N)}^2\leq\frac{1}{2}\|f-\mu^N(f)\|_{L^2(\mu^N)}^2.
 \end{align}
 In fact, for $f$ satisfying $\mu^N(f)=0$ and $\mu^N(f^2)=1$, we have
\begin{align*}
4\times17/16&\geq \frac{1}{2}\{\mu^N(|P_T^N (1+f)|^4)+\mu^N(|P_T^N (1-f)|^4)\}\\
&=1+6\mu^N((P_T^Nf)^2)+\mu^N((P_T^Nf)^4)\geq 1+6\mu^N((P_T^Nf)^2)+(\mu^N((P_T^Nf)^2))^2.
\end{align*}
This implies \eqref{gap}.
Now, set $g_N=f-\mu^N(f)$, then $f=g_N+\mu^N(f)$. Applying $\mu^N((P_{2T}^Ng_N)^3)\leq \left(\mu^N((P_{2T}^Ng_N)^2)\mu^N((P_{2T}^Ng_N)^4)\right)^{1/2}$ and the semigroup property, we have
\begin{align*}
&\mu^N(|P_{2T}^N f|^4)\\
&=(\mu^N(f))^4+4\mu^N((\mu^N(f))^3P_{2T}^Ng_N)\\
&+6\mu^N((\mu^N(f))^2(P_{2T}^Ng_N)^2)+4\mu^N((\mu^N(f))(P_{2T}^Ng_N)^3)+ \mu^N((P_{2T}^Ng_N)^4)\\
&\leq (\mu^N(f))^4+\frac{3}{2}(\mu^N(f))^2\mu^N(g_N^2)+\frac{\sqrt{17}}{4}(\mu^N(f))(\mu^N(g_N^2))^{3/2}+\frac{17}{64}(\mu^N(g_N^2))^2\\
&\leq (\mu^N(f))^4+2(\mu^N(f))^2\mu^N(g_N^2)+\frac{51}{64}(\mu^N(g_N^2))^2\leq ((\mu^N(f))^2+(\mu^N(g_N^2)))^2=(\mu^N(f^2))^2.
\end{align*}
This means
\begin{align}\label{p3t}\sup_{N}\|P_{2T}^N\|_{L^2(\mu^N)\to L^4(\mu^N)}=1.
\end{align} So, by \cite[Proposition 2.3]{FYW2017}, we get
\begin{align}\label{encon}\mu^N(P_{2T}^N(f^2) \log P_{2T}^Nf^2)\leq \frac{2}{3}\mu^N(f^2 \log f^2),\ \ \mu^N(f^2)=1,
\end{align}
and
\begin{align}\label{jocon}\mu^N((P_{2T}^N)^\dagger(f^2) \log( P_{2T}^N)^\dagger f^2)\leq \frac{1}{2}\mu^N(f^2 \log f^2),\ \ \mu^N(f^2)=1.
\end{align}

\subsubsection{Uniform in $N$ $\log$-Sobolev inequality for mean field particle system}
It follows from \eqref{Lipsc3} that the $L^1$-estimate holds, i.e.
\begin{align}\label{kpe}|\nabla P_{t}^N f|\leq \e^{(\sqrt{2K_b^2+2 K_I^2}+1) t}P_{t}^N |\nabla f|.
\end{align}
By taking derivative in $t$ of $\mu^N(P_{t}^Nf^2 \log P_{t}^Nf^2)$ for $f\in C_0^\infty((\R^{2d})^N)$ with $\mu^N(f^2)=1$, using \eqref{eup}, \eqref{kpe} and the Cauchy-Schwartz inequality, it is standard to derive
\begin{align*}&\mu^N(f^2 \log f^2)-\mu^N(P_{t}^Nf^2 \log P_{t}^Nf^2)\\
&\leq  \delta_1\int_0^{t}\mu^N\left(\frac{|\nabla P_s^N f^2|^2}{P_s^N f^2}\right)\d s\\
&\leq  4\delta_1\int_0^{t}\e^{2(\sqrt{2K_b^2+2 K_I^2}+1) s}\mu^N\left(\frac{(P_s^N (|f||\nabla f|))^2}{P_s^N f^2}\right)\d s\\
&\leq  4\delta_1\int_0^{t}\e^{2(\sqrt{2K_b^2+2 K_I^2}+1) s}\mu^N(|\nabla f|^2)\d s,\ \ \mu^N(f^2)=1, f\in C_0^\infty((\R^{2d})^N).
\end{align*}
Combining this for $t=2T$ with \eqref{encon}, we derive the uniform in $N$ log-Sobolev inequality:
\begin{align}\label{unlog}\nonumber&\mu^N(f^2 \log f^2)\\
&\leq  12\delta_1\int_0^{2T}\e^{2(\sqrt{2K_b^2+2 K_I^2}+1) s}\d s\mu^N(|\nabla f|^2),\ \ \mu^N(f^2)=1,f\in C_0^\infty((\R^{2d})^N),
\end{align}
which implies  Talagrand inequality:
\begin{align}\label{Talun}\W_2(\mu,\mu^N)^2\leq \mathbf{c}_{Tal}\mathrm{Ent}(\mu|\mu^N),\ \ \mu\in\scr P_2((\R^{2d})^N).
\end{align}
With the above preparations in hand, we are now in the position to complete the proof of Theorem \ref{cty12pn}.

\begin{proof}[Proof of Theorem \ref{cty12pn}] By \eqref{p3t} and \eqref{unlog} , it remains to prove \eqref{w2conpn} and \eqref{entcopn}.
Reasoning as in the proof of Theorem~\ref{cty}, unless explicitly mentioned otherwise, constants appearing in the proofs are all by default independent from $K_I$ small enough (and, of course, from  $N$).

Firstly, it is known that
$$\mu^N((P_{t}^N)^\dagger f^2\log( P_{t}^N)^\dagger f^2)\leq \mu^N(f^2 \log f^2),\ \ \mu^N(f^2)=1, t\geq 0.$$
This combined with \eqref{jocon} implies that we can  find some constants $c_0,\lambda_0>0$ independent of $N$ such that
\begin{align}\label{NNE}\mathrm{Ent}((P_{t}^N)^\ast\nu^N|\mu^N)\leq c_0\e^{-\lambda_0 t}\mathrm{Ent}(\nu^N |\mu^N),\ \ \nu^N\in\scr P_2((\R^{2d})^N),
\end{align}
which is \eqref{entcopn}.
By \eqref{Lipsc3}, the log-Harnack inequality
$$\mathrm{Ent}((P_{t}^N)^\ast\nu^N|(P_{t}^N)^\ast\bar{\nu}^N)\leq ct^{-3}\W_2(\nu^N,\bar{\nu}^N)^2,\ \ t\in(0,1], \nu^N,\bar{\nu}^N\in\scr P_2((\R^{2d})^N)$$
holds for some constant $c$ independent of $N$, see for instance \cite{RW}. Applying this inequality for $\bar{\nu}^N=\mu^N$ and combining with \eqref{Talun}, the semigroup property and \eqref{NNE}, we have
\begin{align*}\W_2((P_{t}^N)^\ast\nu^N,\mu^N)^2&\leq\mathbf{c}_{Tal}\mathrm{Ent}((P_{t}^N)^\ast\nu^N|\mu^N)\\
&\leq \mathbf{c}_{Tal}c_0\e^{-\lambda_0 (t-1)}\mathrm{Ent}((P_{1}^N)^\ast\nu^N|\mu^N)\\
&\leq \tilde{ c}\mathbf{c}_{Tal}c_0\e^{-\lambda_0 (t-1)}\W_2(\nu^N,\mu^N)^2,\ \ t\geq 1.
\end{align*}
This together with $\sup_{t\in[0,1]}\W_2((P_{t}^N)^\ast\nu^N,\mu^N)^2\leq \hat{c}\W_2(\nu^N,\mu^N)^2$ due to \eqref{Lipsc3} implies \eqref{w2conpn}. Hence, the proof is completed.
\end{proof}

\end{document}